\documentclass[12pt]{amsart}
\usepackage{amsmath,amsthm,amscd,amsfonts,amssymb,epic,eepic,bbm,longtable,tikz-cd,comment, mathtools, mathdots}
\usepackage[pagebackref,colorlinks=true,linkcolor=blue,citecolor=blue]{hyperref}
\usepackage[new]{old-arrows}

\allowdisplaybreaks
\newtheorem{thm}{Theorem}[section]
\newtheorem{question}[thm]{Question}
\newtheorem{cor}[thm]{Corollary}

\newtheorem{lem}[thm]{Lemma}

\newtheorem{prop}[thm]{Proposition}
\theoremstyle{remark}

\theoremstyle{plain}

\newcounter{remarkscounter}

\numberwithin{equation}{section}
\newcommand{\A}{\mathbb{A}}
\newcommand{\GL}{\mathrm{GL}}

\newcommand{\SL}{\mathrm{SL}}
\newcommand{\ZZ}{\mathbb{Z}}

\newcommand{\lto}{\longrightarrow}

\newcommand{\CC}{\mathbb{C}}
\newcommand{\RR}{\mathbb{R}}

\newcommand{\nn}{\mathfrak{n}}
\newcommand{\g}{\mathfrak{g}}
\newcommand{\quash}[1]{}

\theoremstyle{definition}
\newtheorem{defn}[thm]{Definition}

\renewcommand{\bar}{\overline}

\numberwithin{equation}{subsection}

\newenvironment{psmatrix}
  {\left(\begin{smallmatrix}}
  {\end{smallmatrix}\right)}

\usepackage{xifthen}

\newcommand{\optionarg}[1][]{%
  \ifthenelse{\isempty{#1}}%
    {}
    {{#1}}
}

\newcommand{\borel}{\mathfrak{b}}

\newcommand{\BorelC}{B(\CC)}
\newcommand{\GSch}{G}
\newcommand{\TSch}{T}
\newcommand{\gC}{\g}
\newcommand{\GS}{\mathrm{GS}}
\newcommand{\GC}[1][]{G_{\optionarg{#1}}(\CC)}

\newcommand{\GR}[1][]{G_{\optionarg{#1}}(\RR)}

\newcommand{\kC}{\mathfrak{k}}
\newcommand{\KC}[1][]{K_{\optionarg{#1}}(\CC)}

\newcommand{\KR}[1][]{K\optionarg{#1}(\RR)}
\newcommand{\Lbundle}{\mathcal{L}}
\newcommand{\SimpleWeyl}{S}
\newcommand{\tC}{\mathfrak{t}}
\newcommand{\TC}{T(\CC)}

\newcommand{\TR}{T(\RR)}
\newcommand{\Vconvex}{\mathcal{V}}
\newcommand{\Weyl}{W}

\begin{document}

\title{Nonzero $\nn$ cohomology of TDLDS representations}

\author{Jin  Lee}
\address{Department of Mathematics\\
Duke University\\
Durham, NC 27708}
\email{jin.lee@duke.edu}


\thanks{The author is partially supported by NSF RTG grant DMS 2231514. 
Any opinions, findings, and conclusions or recommendations expressed in this material are those of the author and do not necessarily reflect the views of the National Science Foundation.
}

\maketitle
\begin{abstract} We show that a totally degenerate limit of discrete series representation admits a choice of $\nn$ cohomology group that is nonvanishing at a canonically defined degree. We then show that these groups satisfy Serre duality. This produces two $\nn$ cohomology groups, each for a totally degenerate limit of discrete series of U(n+1) and U(n), which are nonvanishing at the same degree. This suggests Gan-Gross-Prasad type branching laws for the TDLDS of unitary groups of any rank. We conclude by constructing a intertwining map of TDLDS for $SU(2,1)$ and $SU(1,1)$. This map will vanish on the minimal $K$ type but induce a non-vanishing map of cohomology.
\end{abstract}

\section{Introduction}

Totally degenerate limits of discrete series (TDLDS) are defined by weakening the regularity assumptions that are used to define discrete series. Therefore, it is natural to ask which of the many known properties of discrete series also hold for TDLDS. In particular, it is expected that TDLDS correspond to Galois representations. However, currently known methods provide no robust link between them and arithmetic geometry in any degree of generality.

It was observed by Carayol that there is a link between TDLDS representations and geometry. In particular, TDLDS may be detected in the coherent cohomology of Griffiths Schmid domains with coefficients in a holomorphic line bundle. The domains that can detect TDLDS cannot be Hermitian symmetric. In fact, they do not admit any complex structure as a quasi projective variety \cite{grt2013}. We must therefore study these line bundles in terms of their equivariant structure over a locally homogeneous complex manifold. Under these weakened assumptions, it is less clear how to apply classical tools from algebraic geometry. Our main tool is the cohomology of TDLDS with respect to a Lie algebra $\nn$. It was shown rather suggestively by Carayol that in certain cases these cohomology groups are non-zero. The non-vanishing of these groups have been further explored in the works of Green, Griffiths, Kerr, and Schmid.

Our contribution in this paper is threefold. We will first extend the results of Carayol and Kerr to all reductive groups, using the combinatorial complex of \cite{soergel1997}. We will then draw a possible link between these cohomology groups and the GGP conjectures for unitary groups. Finally, we will provide evidence for this link with an explicit construction of an intertwining map. The map will send the $SU(2,1)$ TDLDS studied by Carayol, to the anti-holomorphic limit of discrete series for $SU(1,1)$. It will induce a non-trivial restriction map of cohomology. We remark that the construction of our intertwining map and the non-vanishing result of cohomology from $SU(2,1)$ to $SU(1,1)$ does not depend on \cite{soergel1997}.

\subsection{Main results}

To state our main results, we need to specify what we mean by $\nn$ cohomology. Let $\GSch$ be a reductive linear algebraic group over $\RR$ with anisotropic maximal torus $\TSch$. Let $\KR$ be a maximal compact subgroup containing $\TR$ and let $\gC=\CC\otimes_\RR \mathrm{Lie} (\GR)$ denote the complexified Lie algebra of $\GR$. Similarly denote $\kC = \CC\otimes_\RR \mathrm{Lie} (\KR)$ and $\tC=\CC\otimes_\RR \mathrm{Lie} (\TR)$. We obtain a root system $\Phi=\Phi(\gC,\tC)$ with compact roots $\Phi_K$. For any choice of positive roots $\Phi^+$, let $\rho$ be the half sum of positive roots and let $\nn$ be the negative root space.

Let $\pi$ be a limit of discrete series. We say that $\pi$ is TDLDS if the parameter in $\tC^\vee$ defining the infinitesimal character is fixed by the action of the Weyl group. The $\nn$ cohomology of the representation $\pi$ is the Lie algebra cohomology $H^*(\nn,\pi)$. These groups admit an action of $\tC$ induced by the adjoint action on $\nn$ and the action of $\gC$ on $V$. In this paper we are concerned with the weight spaces $H^*(\nn,\pi)_{-\lambda}$ of this action with respect to weights $\lambda \in \tC^\vee$. Our main results are as follows:
\begin{thm}[Theorem \ref{thm: dim Q}] \label{thm:nonvanish:intro}
    Given an integral dominant weight $\lambda \in \tC^\vee$, there exists a choice of limit of discrete series representation $\pi$ with infinitesimal character $\lambda+\rho$ such that $H^q(\nn, \pi)_{-\lambda}$ is nonzero at $q=\#(\Phi_K\cap \Phi^+)$.
\end{thm}

The main idea of Theorem \ref{thm:nonvanish:intro} uses the fact that a TDLDS representation can only be nonzero if the simple roots defining it are non-compact. This forces the differentials computed by Soergel to be 0 at the canonical degrees. This result is in fact proven for all non-zero limits of discrete series. We then prove that these cohomology groups satisfy Serre duality:
\begin{thm}[Theorem \ref{thm: Serre duality}] \label{thm:Serre:duality:intro}
    Given the choices in the previous theorem, one has 
    \[H^q(\nn,V)_{-\lambda} \simeq H^{\#\Phi^+-q}(\nn, V^\vee)_{\lambda+2\rho}~.\]
\end{thm}
We note that the proof of this theorem reveals that the Serre duality of $H^*(\nn,V)$ is compatible with the Serre duality of the $\nn\cap \kC$ cohomology of the $K$ types of $V$. The degrees in which these cohomology groups occur for quasi split unitary groups suggest a geometric application of the Gan Gross Prasad conjecture. 
    We conclude by providing evidence in the low rank setting for special unitary groups. This is given in the following theorem:
    \begin{thm}[Theorem \ref{Main Thm: Low rank GGP}] \label{thm: intro rank 3 to 2}
        Suppose $\GR=SU(2,1)$ and $G^\prime(\RR)=SU(1,1)$. Let $\gC$ and $\nn$ be as denoted earlier, and let $\gC^\prime=\CC\otimes_\RR \mathrm{Lie}(G^\prime(\RR))$. Let $\nn^\prime:= \nn\cap \gC^\prime$ with respect to a choice of embedding $G^\prime(\RR)\subset \GR$. Let $\pi$ be the Harish-Chandra module of the unique TDLDS of $SU(2,1)$, and let $\pi^\prime$ be the Harish-Chandra module of an irreducible admissible representation of $SU(1,1)$.
        
        There exists a choice of embedding $G^\prime(\RR) \hookrightarrow G(\RR)$ such that the following are equivalent:
        \begin{itemize}
            \item $\pi^\prime$ is the anti-holomorphic limit of discrete series of $SU(1,1)$, ie. the one whose $K$ types are not annihilated by $\nn^\prime$.
            \item There exists a morphism $J:\pi \to \pi^\prime$ of $\gC^\prime$ modules, such that the following induced map of cohomology
            \begin{align*}
                H^*(\nn,\pi)\to H^*(\nn^\prime,\pi^\prime),
            \end{align*}
            does not vanish at degree $*=1$, and vanishes at all other degrees.
        \end{itemize}
    \end{thm}
    
    The main difficulty in this theorem is that the maximal torus of $SU(2,1)$ acts on TDLDS with infinite dimensional weight spaces. In our construction of $J$, we find the correct infinite dimensional kernel.    
    The intertwining map then sends the minimal $K$ type for $SU(2,1)$ to zero, rather than the minimal $K\cap SU(1,1)$ type for $SU(1,1)$. 
    This contrasts to the approach taken in \cite{hks2025}, where the translation functor is used to define intertwining maps sending minimal $K$ types to minimal $K\cap G^\prime$ types for $G^\prime$ a suitable reductive subgroup. Instead, we rely on the explicit description of the $SU(2,1)$ TDLDS as given by \cite{carayol1998} and \cite{jw1977}. However, we suspect that all results can be recovered from the the minimal $K$ type and the infinitesimal character, following a computation in \cite[Lecture 9]{ggk2013}.

\subsection{Application of the GGP conjectures}

Let us discuss Theorem \ref{thm: intro rank 3 to 2} in a more general setting. Suppose we are given an inclusion of reductive Lie groups $G^\prime(\RR)\subset \GR$. Assume that $T^\prime_\RR:=\TR\cap G^\prime(\RR)$ is an anisotropic maximal torus. Let $\gC^\prime=\CC\otimes_\RR \mathrm{Lie}(\GR)$ denote the Lie algebra of the smaller group. We then give $\gC^\prime$ a choice of positive roots, where the appropriate assumptions are made so that $\nn^\prime:= \nn\cap \gC^\prime$ is the negative root space. If $\pi$ and $\pi^\prime$ are Harish-Chandra modules for $\GR$ and $G^\prime(\RR)$ respectively, then any morphism $J:\pi \to \pi^\prime$ of $\gC^\prime$ modules will induce a restriction map of cohomology
\begin{align}
    \label{eq: intro local GGP restriction map}
    H^*(\nn,\pi) \to H^*(\nn^\prime,\pi^\prime).
\end{align}

    A global analog of this $\nn$ cohomology restriction map is given in Section \ref{section: GGP applications}. In that section, we construct holomorphic line bundles $\Lbundle_{\lambda_m}$ over a family of compact complex analytic manifolds $\GS(\mathrm{U}_m)$, in terms of unitary groups of ranks $m=n,n+1$. The resulting cohomology $H^*(\GS(\mathrm{U_{m}}),\Lbundle_{\lambda_{m}})$ is given in terms of the $\nn$ cohomology groups discussed above. Motivated by the GGP conjectures, we construct a restriction map
    \begin{align}\label{eq: intro global GGP restriction}
        H^*(\GS(\mathrm{U_{n+1}}),\Lbundle_{\lambda_{n+1}})
        \longrightarrow
        H^*(\GS(\mathrm{U_{n}}),\Lbundle_{\lambda_{n}})~.
    \end{align}
    It is then natural to ask the following question:
    \begin{question} \label{question: GGP Griffiths Schmid}
        Is the restriction map \eqref{eq: intro global GGP restriction} non-zero?
    \end{question}
    
    If this is true, then \eqref{eq: intro global GGP restriction}  can be thought of as a geometric manifestation of the period integral occuring in the GGP conjectures. As we explain in Section \ref{section: GGP applications}, our work implies that in the Archimedean place there is no cohomological obstruction to the nonvanishing of this period. Theorem \ref{thm: intro rank 3 to 2} provides the basic example of how we may attempt to study this question in the low rank case.
    However, fully placing Theorem \ref{thm: intro rank 3 to 2} in the context of the GGP conjectures will require strengthening the result from special unitary groups to unitary groups. This is why \ref{question: GGP Griffiths Schmid} has been framed as a question rather than a conjecture.
    
    If \ref{eq: intro global GGP restriction} is false outside of the low rank case, one can ask if the embedding of Griffiths Schmid varieties in the GGP setting relate to the vanishing of $L$-functions and their derivative, in the sense of the Beilinson Bloch conjectures. This may involve a motivic construction from the Griffiths Schmid varieties that would produce a height pairing. Perhaps such a pairing can relate a diagonal embedding of Griffiths Schmid domains to the vanishing of the derivative of $L(s, \pi \boxtimes\pi^\prime)$ at a central value.

\subsection{Outline}

We begin by defining the representations and cohomology groups of interest in section \ref{section: Notation}. We then consider the combinatorial constructions in section \ref{section: root system definitions}.

Our main results are stated in Sections 
\ref{subsection: Nonvanishing theorem} and
\ref{subsection: Serre duality}
where we give the nonvanishing result in 
Theorem \ref{thm: dim Q} 
and Serre duality in 
Theorem \ref{thm: Serre duality}. The specific degrees in which these cohomology groups occur in the case of unitary groups is given in section \ref{section: GGP applications}. We then relate our results to a cohomological obstruction of the nonvanishing of the period integrals occuring in the GGP conjectures. We conclude by defining a cohomological intertwining map for $SU(1,1)\subset SU(2,1)$ in Section \ref{section: low rank GGP}.

\subsection*{Acknowledgements}
The author thanks Jayce Getz for suggesting the research topic and its connections to the Gan Gross Prasad conjecture. Much of section \ref{section: GGP applications} in particular bears his ideas and patient guidance.
Many critical tools and examples were brought to the author's attention by Matt Kerr, Alfio Fabio La Rosa, and Colleen Robles. We also thank Wolfgang Soergel for the numerous discussions that helped with understanding key details in his results.
We thank Sugwoo Shin with help on answering questions on tempered $L$-packets of unitary groups. 
Finally, we thank Michael Harris and Birgit Speh for the discussions that led to the intertwining map that we give in Section \ref{section: low rank GGP}.


\section{Notation} \label{section: Notation}

Let $\GSch, \TSch, \TR\subset \KR\subset \GR$ and $\tC\subset \kC\subset \gC$ be as in the previous section. Let $\Phi:=\Phi(\gC,\tC)$ be the set of roots, and $\Phi_K$ be the set of compact roots. We recall that $\Phi_K$ can be defined as the roots that occur in the root space decomposition 
\[
    \kC=\tC\oplus (\bigoplus_{\alpha\in \Phi_K} \gC_\alpha)
\]
of $\kC$ by the action of $\tC$.
The Weyl group shall be denoted $\Weyl:=\Weyl(\gC,\tC)$, and the compact Weyl group will be denoted $\Weyl_K\subset \Weyl$.

Let $\Phi^+\supset \Sigma$ be a choice of positive roots and simple roots, let $\rho$ be the half sum of positive roots, and write the positive Weyl chamber with respect to this choice as
\[
    C=\{\mu\in \tC^\vee | (\mu,\alpha)>0 \text{ for all } \alpha\in \Phi^+\} ~.
\]
Recall that $C$ gives a fundamental domain of the action of $W$ on $\tC^\vee$.

Let $\nn=\oplus_{\alpha} \gC_{-\alpha}$ be the sum of negative root spaces giving the Borel subalgebra $\borel=\tC\oplus \nn$. We get a corresponding Borel subgroup with Levi decomposition $\BorelC=\TC\cdot N$. We denote by $U(\cdot)$ the universal enveloping algebra. Given any $\lambda+\rho \in \tC^\vee$, we then write the infinitesimal character $\chi_{\lambda+\rho}\in Z(\gC)^\vee$ according to the Harish-Chandra isomorphism.
Note that $\chi_{w(\lambda+\rho)}$ gives rise to the same character for any $w\in \Weyl$. 

\subsection{Parametrization of Limits of Discrete Series, and n Cohomology} \label{section: Parametrization of LDS}

We now describe all limits of discrete series $\pi=\pi(Q,\chi_{\lambda+\rho})$ in terms of parameters $Q$ and $\chi_{\lambda+\rho}$ as follows.

Let $X$ be the flag variety of all Borel subalgebras of the complexification $\gC$.  Let $\iota: Q\hookrightarrow X$ be the inclusion of a closed $\KC$ orbit $Q$, obtained from the natural action of $\GC$ on $X$. We obtain a discrete series representation $\Gamma(\iota_* \mathcal{O}_Q)$ with infinitesimal character $\chi_{0+\rho}$. Here, $\Gamma(\iota_* \mathcal{O}_Q)$ is the global sections of the $\mathcal{D}$ module obtained by pushing forward the structure sheaf on $Q$. Let $\lambda\in \tC^\vee$ be an integral character, ie. corresponding under the exponential map to a well defined character $\exp(\lambda): \TC \to \CC^\times$ of the maximal torus. Let $T^{\lambda+\rho}$ denote the Zuckerman translation functor that sends representations with central characters $\chi_{\rho}$ to those with $\chi_{\lambda+\rho}$. We then define
\[
    \pi(Q,\chi_{\lambda+\rho}) := T^{\lambda+\rho} \Gamma(\iota_* \mathcal{O}_Q) ~.
\]
By convention, we say that a TDLDS representation is a representation of the form $\pi(Q,\chi_{\lambda+\rho})$ when $\lambda+\rho$ is orthogonal to all roots $\Phi$. In the semisimple case this condition means that $\lambda+\rho=0$ and such representations are classified in \cite{Carayol_Knapp}. In the case of reductive groups, the maximal torus $\TR$ will contain the center of $\GR$ which can act by a nontrivial character.

Let $H^\bullet(\nn,\pi(Q,\chi_{\lambda+\rho}))$ denote the Lie algebra cohomology of $\nn$ with coefficients in $\pi(Q,\chi_{\lambda+\rho})$. These cohomology groups admit an action of $\tC$, induced by the adjoint action of $\tC$ on $\nn$, and the restriction to $\tC$ of the $\gC$ action on $\pi(Q,\chi_{\lambda+\rho})$. We will denote $H^\bullet (\nn,\pi(Q,\chi_{\lambda+\rho}))_{-\lambda^\prime}$ to be the $-\lambda^\prime$ weight spaces of this action for $\lambda^\prime \in \tC^\vee$.

Our primary interest in this paper is to prove nonvanishing of the cohomology groups
\[
    H^\bullet (\nn,\pi(Q,\chi_{\lambda+\rho}))_{-\lambda^\prime} ~.
\]
By the Casselman Osborne Lemma \cite{co1975}, these cohomology groups are $0$ unless $\lambda^\prime = w(\lambda+\rho)-\rho$ for some element of the Weyl group. Note that for TDLDS, the Weyl group fixes $\lambda+\rho$ since we have assumed it is orthogonal to all roots. Hence, the $\lambda^\prime$ weight spaces of our $\nn$ cohomology groups will be 0 unless $\lambda^\prime = \lambda+\rho-\rho = \lambda$. Therefore, we are reduced to studying
\[
    H^\bullet (\nn,\pi(Q,\chi_{\lambda+\rho}))_{-\lambda} ~.
\]

\section{Combinatorial constructions and non-vanishing theorems}
\label{section: root system definitions}

\subsection{Bruhat order and relative Bruhat length}
On the Weyl group $\Weyl$ we want to use our choice of positive roots to define the Bruhat order. We will denote reflections as $\sigma_{\alpha} \in W$ for a given root $\alpha \in \Phi$. Our positive root system gives us a choice of simple roots $\Sigma$ and therefore a set of simple reflections $\SimpleWeyl=\SimpleWeyl(\Sigma)$. We use this $\SimpleWeyl$ to define a Bruhat length $l(x)$ on any $x\in \Weyl$ as the minimum number $n$ of simple reflections $\sigma_{\alpha_1},...,\sigma_{\alpha_n}\in \SimpleWeyl$ such that
\[
    x=\sigma_{\alpha_1}\sigma_{\alpha_2}...\sigma_{\alpha_n} ~,
\]
where $\alpha_i\in \Sigma$ is a simple root.

We can now define the Bruhat order. For any $x,y\in \Weyl$, let us write $x\to y$ whenever $x=s_\alpha y$ for some $\alpha\in \Phi$ and $l(x)=l(y)+1$. We warn the reader that $s_\alpha$ need not be a simple reflection. The Bruhat order $(\Weyl,\leq)$ is then defined on $\Weyl$ by 
\[
    x<y \iff x\to x_1 \to ... \to x_n \to y
\]
for some sequence $x_1,...,x_n$ of elements in $\Weyl$. We shall refer to $(\Weyl,\rightarrow)$ as the Hasse diagram of the Bruhat order.

The chain complexes computing our $\nn$ cohomology groups also require us to work with other notions of lengths.
We first note that $l(x)$ is the number of Weyl chamber walls which separate the $\Phi^+$ positive Weyl chamber $C$ from $xC$, which by \cite{humphreys1990} is equal to: 
\[
    l(x)=\#\{\alpha\in \Phi^+ | (xC,\alpha)<0\}=\# (\Phi^+\cap x \Phi^-) ~.
\]

In a similar manner, we can define as in \cite{soergel1997} a compact Weyl length $l_K(x)$ as
\begin{align*}
    l_K(x):=\#\{\alpha\in \Phi^+\cap \Phi_K | (xC,\alpha)< 0\}=\#(\Phi^+\cap x\Phi^- \cap \Phi_K) ~,
\end{align*}
where $(xC,\alpha) < 0$ means $(x\beta,\alpha) < 0$ for all $\beta\in C$.

\subsection{Chain complex on subsets of the Weyl group}
\newcommand{\axzeroortwo}{$\star$}

    Let $\Vconvex\subset \Weyl$ be a subset of the Weyl group, and $(\Vconvex,\rightarrow)$ the induced subgraph from $(\Weyl,\rightarrow)$. We construct chain complexes $(C^\bullet \Vconvex,d)$ using the Bruhat length to define the grading, and the Bruhat order to define differentials. This will be our main ingredient to compute $\nn$ cohomology via Soergel's formula.
    
    In order for our construction of differentials to square to zero, we impose on our subset $\Vconvex$ the following condition:
    \begin{align*}
        \text{If } w_1,w_2\in \Vconvex 
        \text{ are such that } l(w_1)-2=l(w_2)
        \text{ then}
        \\
        \{w\in W | w_1 < w < w_2\} 
        \text{ has two elements or is empty.}
        \tag{\axzeroortwo}
        \label{ax: zeroortwo}
    \end{align*}

Given such a set $\Vconvex\subset \Weyl$ satisfying \eqref{ax: zeroortwo}, let $C^\bullet \Vconvex$ be the free $\CC$ vector space generated by the basis $\Vconvex$. We endow this with a graded structure coming from the length function, such that for any $q\in \mathbb{Z}$,
\[
    C^q\Vconvex=Span_\CC \{v \in \Vconvex : l(v)=q\} ~.
\]

The graph $(\Vconvex,\rightarrow)$ can now be used to define a differential on $C^\bullet \Vconvex$ as follows. Write $d$ as a matrix with entries $s_\Vconvex(x,y)$ indexed by basis elements $x,y\in \Vconvex$. We assume these matrix entries satisfy the following desiderata:

\begin{defn}
    Let $d: C^\bullet \Vconvex \to C^{\bullet+1}\Vconvex$ be defined by
    \[
        dx = \sum_{y\in \Vconvex} s_\Vconvex(x,y)y \text{ for any } x\in \Vconvex
    \]
    where we choose $s(x,y)$ for each $x,y\in \Vconvex$ such that
    \begin{enumerate}
        \item $s_\Vconvex(x,y) \in \{\pm1\}$ if $x\rightarrow y$ and \\$s_\Vconvex(x,y)=0$ otherwise
        \item if $x\rightarrow w \rightarrow y$ and $x\rightarrow w^\prime \rightarrow y$ for $w\neq w^\prime$ then 
        \[  s_\Vconvex(x,w)s_\Vconvex(w,y)+s_\Vconvex(x,w^\prime)s_\Vconvex(w^\prime,y)=0 ~.
        \]
    \end{enumerate}
    \label{ax: conditions for differentials}
\end{defn}

The first condition asks for $d$ to be an adjacency matrix for the graph $(\Vconvex,\rightarrow)$ where some of its entries $s_\Vconvex(x,y)$ may be $-1$ instead of $1$. The second condition asks for those negative entries to be chosen such that $d^2=0$.

It is shown in \cite{bgg1975} that when $\Vconvex=\Weyl$, such a choice of $s_\Weyl(x,y)$ is always possible. Suppose instead we are given a subset $\Vconvex\subset \Weyl$ that satisfies \eqref{ax: zeroortwo}. In this case, we can obtain a choice of matrix coefficients $s_\Vconvex$ satisfying the conditions of Definition \ref{ax: conditions for differentials} by setting $s_\Vconvex(x,y)=s_\Weyl(x,y)$ for any $x,y\in \Vconvex$. Hence, we see that \eqref{ax: zeroortwo} implies the existence of a choice of $s_\Vconvex$ that satisfies Definition \ref{ax: conditions for differentials}, and in fact the converse also holds. Since $d^2=0$, we can take the homology of this chain complex, and the result does not depend on the choice of $s_\Vconvex$ as is shown in \cite{soergel1997}.

\subsection{Soergel's formula}

We can now state the following theorem with which we can compute $\nn$ cohomology:
\begin{thm}[\cite{soergel1997}]
    Let $Q$ be a closed $\KC$ orbit in $X$. Let $A=\{w\in \Weyl \mid w(\nn\oplus \tC)\in Q \}$ and $D=\{w\in \Weyl \mid \lambda+\rho \in w\bar{C}\}$. Then,
    \[
        H^q (\nn, \pi(Q,\chi_{\lambda+\rho}))_{-\lambda} = \bigoplus_{c\in \ZZ} H^{2c+\dim_\CC Q - q} \Vconvex(A,D,c) ~,
    \]
    where $\Vconvex(A,D,c)=
        \{w \in A \cap D \mid l(x)=l_K(x)+c\}$.
    \label{soergel's theorem}
\end{thm}

\subsection{Nonvanishing theorem}
\label{subsection: Nonvanishing theorem}

We now prove the following nonvanishing theorem:

\begin{thm} \label{thm: dim Q}
    Let $\lambda+\rho \in \bar{C}$ be an integral dominant but possibly singular character.
    Let $Q\subset X$ be the closed $K$ orbit containing $\borel=\tC\oplus \nn$ with $\nn$ the negative root space.
    Suppose that the limit of discrete series $\pi(Q,\chi_{\lambda+\rho})$ is not zero. We then have 
    \[
        H^{\dim_\CC Q}(\nn,\pi(Q,\chi_{\lambda+\rho}))_{-\lambda} \neq 0
    \]
\end{thm}

\begin{proof}
    The identity element $id\in \Weyl$ satisfies $l(id)=l_K(id)=0$ and therefore $id \in \Vconvex(A,D,0)$. We want to show the differential vanishes on this element. If $s\in \Vconvex(A,D,0)$ has length $l(s)=1$ then it is a simple reflection $s=s_\alpha$ for $\alpha$ a simple root. Furthermore, $s\in D$ fixes the infinitesimal character, ie. $\alpha$ is orthogonal to $\lambda+\rho$. If such a root is compact, then the limit of discrete series thus defined will equal zero, a contradiction. This proves that $\Vconvex(A,D,0)$ cannot contain any element of length 1. 

    We thus get a vanishing differential $d: C^0 \Vconvex(A,D,0) \to C^1 \Vconvex(A,D,0)=0$ at degree zero. 
    Therefore, we get a 1 dimensional contribution to $H^0 \Vconvex(A,D,0)$ coming from the identity element of the Weyl group. 
    Going back to \ref{soergel's theorem} gives us
    \[
        H^{\dim_\CC Q}(\nn,\pi(Q,\chi_{\lambda+\rho}))_{-\lambda}
        =H^0 \Vconvex(A,D,0) \oplus \bigoplus_{c\geq 1} H^{2c} \Vconvex(A,D,c) ~.
    \]
    We have just shown that the $c=0$ term above gives a one dimensional contribution to $\nn$ cohomology.
\end{proof}

As an application, note that if $\pi(Q,\chi_{\lambda+\rho})$ is a TDLDS then $
    H^{\dim_\CC Q}(\nn, \pi(Q,\chi_{\lambda+\rho}))_{-\lambda}\neq 0 ~.
$


\subsection{Serre duality}
\label{subsection: Serre duality}

Set $N=\#(\Phi^+)$. Let us first note the following proposition.
\begin{prop}
    Let $w_0\in \Weyl$ be the longest Weyl element. The following are true:
    \begin{enumerate}
        \item $l(aw_0)=N-l(a)$
        \item $l_K(aw_0)=\#(\Phi^+\cap \Phi_K)-l_K(a)$
        \item If $a\to b$ in $\Weyl$ then $aw_0 \leftarrow bw_0$
    \end{enumerate}
    \label{properties of w_0}
\end{prop}
\begin{proof}
    Note that $w_0C=-C$ so we immediately get (1) and (2).
    The last statement also follows, since $a \to b$ implies $b=s_\alpha a$ for a reflection $s_\alpha$ with $\alpha\in \Phi^+$, and $l(aw_0)=N-l(a)=N-l(b)-1=l(bw_0)-1$ and therefore $aw_0 \leftarrow bw_0$.
\end{proof}

Our current goal is to provide a statement of Serre duality between the cohomology groups $H^\bullet \Vconvex$ and $H^\bullet (\Vconvex w_0)$ as defined in section \ref{section: root system definitions}. In Definition \ref{ax: conditions for differentials} we associated a chain complex $C^\bullet \Vconvex$ to any subset $\Vconvex\subset \Weyl$ of the Weyl group satisfying condition \eqref{ax: zeroortwo}. Getting to a duality statement will require us to construct the differential for $\Vconvex w_0$ using a choice of differential on $\Vconvex$.

Let $\Vconvex\subset \Weyl$ be a subset satisfying \eqref{ax: zeroortwo} and fix a differential $d_\Vconvex$ of $C^\bullet \Vconvex$. We then let $d_{\Vconvex w_0}$ be given by $s_{\Vconvex w_0}(xw_0,yw_0) = s_\Vconvex(y,x)$. By Proposition \ref{properties of w_0}, if $u w_0$ and $v w_0$ are two distinct elements in $\{w w_0 \in \Weyl | x w_0 \to w w_0 \to y w_0\}$ then
\begin{align*}
    &s_{\Vconvex w_0}(yw_0,uw_0)s_{\Vconvex w_0}(uw_0,xw_0)+s_{\Vconvex w_0}(yw_0,vw_0)s_{\Vconvex w_0}(vw_0,xw_0)
    \\
    &=s_{\Vconvex}(x,u)s_{\Vconvex}(u,y)+s_{\Vconvex}(x,v)s_{\Vconvex}(v,y)=0 ~,
\end{align*}
so it follows that $d_{\Vconvex w_0}$ satisfies the conditions in Definition \ref{ax: conditions for differentials} for differentials of $C^\bullet (\Vconvex w_0)$.

The following duality lemma will be the main ingredient to our proof of Serre duality for $\nn$ cohomology.
\begin{lem}
    $H^{k} \Vconvex= (H^{N-k} (\Vconvex w_0))^\vee$
    \label{w_0 duality} ~.
\end{lem}
\begin{proof}
    Let 
    $B: C^k \Vconvex \times C^{N-k} (\Vconvex w_0) \to \CC$ be the symmetric nondegenerate bilinear form defined on the basis elements $a\in \Vconvex$ and $bw_0\in \Vconvex w_0$ as
    \[
        B(a,bw_0)=\delta_{a,b}~,
    \]
    where $\delta_{a,b}$ denotes the Kronecker delta function.
    From $B$, we get a map $B^\prime: C^k \Vconvex \to (C^{N-k}(\Vconvex w_0))^\vee$ given by $B^\prime(a)(bw_0)=B(a, bw_0)$. Note that $B^\prime$ is a map of chain complexes graded by $k$ since
    \[
        B(d_\Vconvex a, bw_0)
        =B(s_\Vconvex(b,a) b,bw_0)
        =s_\Vconvex(b,a)
        =s_{\Vconvex w_0}(aw_0,bw_0)
        =B(a,d_{\Vconvex w_0}bw_0) ~.
    \]
    Since $B$ is nondegenerate, $B^\prime$ must also be an isomorphism.

    We then claim that the
    $k$th cohomology of $C^{N-\bullet}(\Vconvex w_0)^\vee$ is isomorphic to $H^{N-k} (\Vconvex w_0)^\vee$. Indeed, we have
    \begin{align*}
        H^{k}(C^{N-\bullet}(\Vconvex w_0)^\vee, d^\vee)
        =\frac{Ker(d^\vee)}{Im(d^\vee)}
        =\frac
        {\{v^\vee \in (C^{N-k})^\vee : v^\vee|_{Im (d)} =0\}}
        {\{v^\vee \in (C^{N-k})^\vee : v^\vee|_{Ker(d)}=0 \}}
        \simeq H^{N-k}(\Vconvex w_0)^\vee ~.
    \end{align*}
    
    We have thus proven that
    \[
        H^k \Vconvex 
        \overset{\sim}{\longrightarrow}
        H^{k}(C^{N-\bullet}(\Vconvex w_0)^\vee)
        \simeq
        H^{N-k} (\Vconvex w_0)^\vee
    \]
    is an isomorphism
\end{proof}

We are now ready to prove Serre duality. Let $\Sigma\subset \Phi^+\subset \Phi$ and $\Phi_K\subset \Phi$ be as before. Let $Q\subset X$ be an arbitrary closed $\KC$ orbit with $Q^{op}:=w_0 Q$. Let $\chi_{\lambda+\rho}$ be an arbitrary character, with $\lambda^{op}=w_0(\lambda+\rho)-\rho$.

\begin{thm} \label{thm: Serre duality}
    Suppose $\lambda^{op}=-\lambda-2\rho$ and $Q^{op}$ is the $\KC$ orbit of $x w_0 \borel$ where $x \borel \in Q$. This implies $Q^{op}$ is closed.
    Let $\pi=\pi(Q,\chi_{\lambda+\rho})$ and $\pi^{op}=\pi(Q^{op},\chi_{\lambda^{op}+\rho})$.
    We obtain
    \[
        \dim_\CC H^{N-k}(\nn,\pi)_{-\lambda}
        =
        \dim_\CC H^{k}(\nn,\pi^{op})_{-\lambda^{op}} ~.
    \]
\end{thm}
\begin{proof}
    The fact that $Q^{op}$ is closed follows from \cite[Proposition 7.1.1 Statement 3]{fhw2006} applied to the fact that if a Cartan involution preserves a Borel subalgebra, then it must also preserve the opposite Borel.
    
    Let us write $A^{op}=\{x \in \Weyl | x\borel \in Q^{op}\}$
    and $D^{op}=\{x \in \Weyl | x^{-1}(\lambda^{op}+\rho)\in \bar{C}\}$. We first show that $A^{op}=A w_0$ and $D^{op}=D w_0$, ie. both sets can be obtained by right multiplying the Weyl element.
    
    By construction, there is some $x_0\borel \in Q$ such that $x_0w_0\borel \in Q^{op}$, and both $Aw_0$ and $A^{op}$ are cosets of $\Weyl_{K}$ of the same  element $xw_0$ so this implies $A^{op} = Aw_0$.
    On the other hand, if $x\in D$ then recall that $x^{-1}(\lambda+\rho)\in \bar{C}$ by definition. Since $w_0$ sends the Weyl chamber $C$ to $-C$, we obtain
    \[
        (x w_0)^{-1} (\lambda^{op}+\rho)
        =
        w_0 x^{-1}(-\lambda-\rho)
        =
        -w_0 x^{-1}(\lambda+\rho)
        \in
        -w_0 \bar{C} = \bar{C} ~.
    \]
    Hence, $xw_0\in D^{op}$ so we have shown that $D^{op}\supseteq D w_0$. 
    
    Conversely, if $xw_0 \in D^{op}$ then note that $(xw_0)^{-1} (\lambda^{op}+\rho) \in \bar{C}$. Applying $w_0$ to the left gives us
    \[
        x^{-1}(\lambda^{op}+\rho) \in w_0\bar{C} = -\bar{C} ~,
    \]
    and therefore we obtain
    \[
        x^{-1}(\lambda+\rho)
        =
        x^{-1}(-\lambda^{op}-\rho)
        =-x^{-1}(\lambda^{op}+\rho)
        \in \bar{C} ~.
    \]
    We have just shown $D^{op}\subseteq Dw_0$ so this proves the claim.

    We now want to show that $\Vconvex(A^{op},D^{op},c)=\Vconvex(A,D,N-\dim Q - c)w_0$. Indeed, note that 
    \[
        \Vconvex(A^{op},D^{op},c)
        =
        \Vconvex(A w_0, D w_0, c) ~.
    \]
    But if $xw_0\in \Vconvex(A w_0, D w_0, c)$ (where we note $xw_0\in Aw_0\cap Dw_0$ and we choose $x$ to be in $A\cap D$)
    then
    \begin{align*}
        c=l(xw_0)-l_K(xw_0)
        =N-l(x)-(\dim Q -l_K(x)) ~.
    \end{align*}
    Rearranging gives us
    $l(x)-l_K(x)=N-\dim Q - c$. This implies $x \in \Vconvex(A,D,N-\dim Q - c)$, so we see that 
    $
        \Vconvex(Aw_0,Dw_0,c)\subset \Vconvex(A,D, N-\dim Q - c) w_0
    $.
    We obtain the reverse inclusion by replacing $A$ and $D$ with $Aw_0$ and $Dw_0$. Now, we have
    \[
        \Vconvex(A^{op},D^{op},c)
        =
        \Vconvex(Aw_0,Dw_0,c)
        =
        \Vconvex(A,D, N-\dim Q - c) w_0 ~.
    \]

    We can now invoke all of our previous statements to finish the proof. We have
    \begin{align*}
        \dim H^{N-k}(\nn, \pi^{op})_{-\lambda}
        \overset{\ref{soergel's theorem}}{=}
        &\dim \bigoplus_c H^{2c+\dim Q - N + k} \Vconvex(A^{op},D^{op},c)
        \\=
        &\dim \bigoplus_c H^{2c+\dim Q - N + k} \Vconvex(A,D,N-\dim Q - c) w_0
        \\
        =
        &\dim \bigoplus_c H^{2c+\dim Q - N + k} \Big(\Vconvex(A,D,N-\dim Q - c) w_0\Big)^\vee
        \\
        \overset{\ref{w_0 duality}}{=}
        &\dim \bigoplus_c H^{N-(2c+\dim Q - N + k)} \Vconvex(A,D,N-\dim Q - c)
        \\
        \overset{(*)}{=}
        &\dim \bigoplus_{c^\prime} H^{2c^\prime + \dim Q - k} \Vconvex(A,D,c^\prime)
        =
        \dim H^k (\nn, \pi)_{-\lambda^{op}}
    \end{align*}
    where in the last equality $(*)$ we have set $c^\prime = N-\dim Q - c$.

\end{proof}

The nonvanishing result and Serre duality tell us that if our infinitesimal character is given by $\lambda+\rho$ orthogonal to all roots, then the following holds:
\begin{cor} \label{cor: easy to state non-vanishing}
    There exist TDLDS representations $\pi$ and $\pi^\prime$ such that
    \[
    \dim_\CC H^{\dim_\CC Q}(\nn, \pi)_{-\lambda}
    =
    \dim_\CC H^{N-\dim_\CC Q}(\nn, \pi^\prime)_{\lambda+2\rho}
    \neq 0 ~.
    \]
\end{cor}
This concludes our proof of the main results.


\section{The GGP conjecture and TDLDS representations} \label{section: GGP applications}

The degrees in which these $\nn$ cohomology groups are nonzero suggest a geometric interpretation of the GGP conjectures. We will first define the line bundles that will serve as geometric avatars of the $\nn$ cohomology groups that were just studied. We then suggest how the GGP conjectures can be used to relate two different line bundles that each detect different representations. We conclude by specializing to the case of unitary groups where the GGP conjectures were proven.


\subsection{Automorphic representations and line bundles over Griffiths-Schmid varieties} \label{def line bundles and GS}

We want to define the Griffiths Schmid manifolds $GS(\GSch)^{K^\infty}$ and clarify how its cohomology can detect TDLDS representations. We will make a few assumptions that allow our Griffiths Schmid manifolds to be compact.
Let $F$ be a totally real number field, and let $G$ be a reductive group over $F$ with anisotropic center.  
%
%
For simplicity we assume that there is a place $v_0|\infty$ such that $G_{F_{v}}$ is anisotropic for $v|\infty,$ $v \neq v_0$.

%
%

Assume there is an anisotropic maximal torus $T \leq G_{F_{v_0}}$. We will denote $\A_F$ to be the adeles of $F$, and if $S$ is a set of places of $F$ then we shall denote $\A_F^S$ to be the adeles away from $S$ and $F_S=\prod_{v\in S} F_v$.
Let $K^{v_0}=\prod_{v \neq v_0} K_v <G(\A_F^{v_0})$ be a subgroup such that $K_{v}=G(F_v)$ for $v|\infty$ with $v \neq v_0$ and $K^\infty = (K^{v_0})^\infty$ is compact and open in $G(\A_F^\infty).$
Upon taking $K^\infty$ sufficiently small we see that  there exist discrete subgroups $\Gamma_i\leq G(F_{v_0})$ and a decomposition of real manifolds
\begin{align}
    G(F) \backslash G(\A_F) / (T(F_{v_0}) \times K^{v_0}) 
    \simeq \sqcup_{j=1}^n \Gamma_j \backslash G(F_{v_0})/ T(F_{v_0}) ~.
    \label{eq: finiteness of class number}
\end{align}
Note that these manifolds are compact.

We can endow this space with a complex structure as follows. Recalling that $F$ is totally real, there is a unique embedding $F_{v_0} \to \CC$. 
Let $B$ be a Borel subgroup of $G \times_{F_{v_0}}\CC$ 
such that the Lie algebra $\borel$ of its complex points satisfies
$
    \borel=\tC \oplus \nn = \tC \oplus \bigoplus_{\alpha\in \Phi^+} \gC_{-\alpha}
$
as in section \ref{section: Notation}.

This Borel must also satisfy $\BorelC\cap G(F_{v_0}) = T(F_{v_0})$, and will not descend to a Borel in the real group $G(F_{v_0})$. We obtain a Borel embedding, ie. an open embedding of
$
    G(F_{v_0})/T(F_{v_0})
$
into the flag variety $\GC/\BorelC$, and we thus obtain a complex structure on $\Gamma_i\backslash G(F_{v_0})/T(F_{v_0})$ for each $i$. Let
\[
    GS(G)^{K^\infty} = G(F) \backslash G(\A_F) / (T(F_{v_0}) \times K^{v_0}) ~.
\]
It is a locally homogeneous space referred to as a Griffiths Schmid manifold. The complex structure is induced in each of its $\Gamma_i$ components by \ref{eq: finiteness of class number} and the Borel embedding. We shall also write 
\[
    \GS(\GSch) = \varprojlim_{K^{\infty}} \GS(\GSch)^{K^\infty} ~,
\]
where the inverse limit is taken over all neat compact open subgroups.

We will now define holomorphic line bundles over these Griffiths Schmid manifolds whose cohomology allows us to study TDLDS representations. Note that any character $\lambda:T(F_v) \to \CC^\times$ gives rise to a character $\lambda: \BorelC \to \CC^\times$ by extending trivially to $\BorelC$. We can then define a line bundle $\GC \times_{\BorelC, \lambda} \CC$ that fibers holomorphically over $\GC/\BorelC$. The global sections correspond to a maps $s: \GC\to \CC$ that satisfy 
$s(gb)=s(g)\lambda(b)$ for any $g\in \GC$ and $b\in \BorelC$. From the Borel embedding, we get a line bundle $\Lbundle_\lambda$ by pullback
\[
    \begin{tikzcd}
        \Lbundle_\lambda 
            \arrow[r]
            \arrow[d]&
        \GC\times_{\BorelC,\lambda} \CC 
            \arrow[d]
        \\
        G(F_{v_0})/T(F_{v_0})
            \arrow[r]&
        \GC/\BorelC
    \end{tikzcd}
\]
from the embedding of an open subset. After taking arithmetic quotients, we obtain a line bundle of $\Gamma_i$ invariant sections
\[
    \begin{tikzcd}
        \Lbundle_\lambda^{\Gamma_i}
            \arrow[r]
        &
        \Gamma_i \backslash G(F_{v_0}) / T(F_{v_0})
    \end{tikzcd}
\]

We shall abuse notation and denote by $\Lbundle_\lambda$ the induced line bundle on $\GS(\GSch)^{K^\infty}$ obtained from \ref{eq: finiteness of class number}, or the induced line bundle over the inverse limit $\GS(\GSch)=\lim_{\leftarrow}\GS(\GSch)^{K^\infty}$, whenever the context is clear.

For every irreducible admissible representation $\pi^\infty$ of $G(\A_F^\infty)$ let $H^*(GS(G),\mathcal{L}_{\lambda})(\pi^\infty)$ be the $\pi^\infty$ isotypic component.
Assume $G$ is anisotropic. We have the following direct sum decomposition of cohomology in the sense of Matsushima's formula \cite[\S 3.2]{carayol1998}
\begin{align}
    H^*(\GS(\GSch),\Lbundle_\lambda)
    =\bigoplus_\pi
        H^*(\nn,\pi_{v_0})_{-\lambda} \otimes \pi^\infty
    \label{eq: cohomology decomposition}
\end{align}
where the sum is over automorphic representations $\pi$ of $G(\A_F)$ up to multiplicity. The representations $\pi$ that occur in this summation will automatically be cuspidal.  We say a cuspidal automorphic representation $\pi$ of $G(\A_F)$ is TDLDS if $\pi_{v_0}$ is a TDLDS representation and $\pi_v$ is the trivial representation for $v|\infty,$ $v \neq v_0$. We remind the reader that $G(F_v)$ is compact for the Archimedean places $v\neq v_0$ away from $v_0$.

Let $\chi_{\pi_{v_0}}$ be the infinitisimal character of $\pi_{v_0}$. Once again, we identify $\chi_{\pi_{v_0}}=\chi_{\lambda+\rho}$ with an element of $\mathfrak{t}^\vee$ up to the action of the Weyl group.

\begin{lem}
    In the decomposition \eqref{eq: cohomology decomposition} the group $H^*(\nn,\pi_{v_0})_{-\lambda}$ vanishes unless $\chi_{\pi_{v_0}}$ is in the Weyl orbit of $\lambda+\rho.$  In particular, if $H^*(\mathfrak{n},\pi_{v_0})_{-\lambda}$ is nonzero and $\lambda+\rho$ is fixed by the Weyl group, then $\pi_{v_0}$ is a TDLDS representation.
\end{lem}
\begin{proof}
    This is a reformulation of \cite[Corollary 2.7]{co1975}
\end{proof}

Now let us restrict attention to unitary groups. Let $E/F$ be a CM extension of number fields, and $V_{n+1} \cong E^{n+1}$ be a vector space equipped with a Hermitian form. Let 
\begin{align}
\mathrm{U}_{n+1}:=\mathrm{U}(V_{n+1})
\end{align} be the unitary group over $F$ that preserves the Hermitian form on $V_{n+1}$.

As above, we assume that there is a place $v_0|\infty$ such that $\mathrm{U}_{n+1,F_{v_0}}$ is quasi-split and $\mathrm{U}_{n+1,F_{v}}$ is anistropic for any Archimedean place $v \neq v_0$.
Let $V_n<V_{n+1}$ be a subspace of codimension $1$ over $E$ such the Hermitian form on $V_{n+1}$ is nondegenerate when restricted to $V_n.$ We assume that $\mathrm{U}_n:=\mathrm{U}(V_n)$ has the property that $\mathrm{U}_{n,F_{v_0}}$ is quasi-split.  Since $\mathrm{U}_{n+1,F_{v}}$ is anisotropic for any Archimedean place $v \neq v_0$, the same holds for $\mathrm{U}_n.$

Let $V_{n+1}$ be another Hermitian vector space that admits an embedding of $V_n$ as a nondegenerate subspace of $V_{n+1}$. We can make appropriate assumptions on the signature of $V_{n+1}$ and the embedding of $V_n$ such that $\mathrm{U}_{n+1}$ is also quasi-split with an embedding $\mathrm{U}_n \hookrightarrow \mathrm{U}_{n+1}.$ 
We assume that there is a place $v_0|\infty$ such that $U_{n+1,F_{v_0}}$ and $U_{n,F_{v_0}}$ are quasi-split, and $U_{n+1, F_{v}}$ and $U_{n,F_v}$ are anistropic for any other Archimedean place $v \neq v_0.$

We choose Borel subgroups $B_n$ and $B_{n+1}$ of the complexifications $\mathrm{U}_n \times_{F_{v_0}}\CC$ and $\mathrm{U}_{n+1} \times_{F_{v_0}}\CC$ compatibly in the sense that $B_{n+1} \cap (\mathrm{U}_{n} \times_{F_{v_0}}\CC) =B_n$. We shall assume each $B_m(\CC)$ contains the complex points $T_m(\CC)\subseteq B_m(\CC)$ of an anisotropic maximal torus $T_m\subseteq \mathrm{U}_{m}$ for $m=n,n+1$. With respect to these choices, we shall denote $\Phi_m, \borel_m, \tC_m, \nn_m$, etc., in the same manner as earlier in this section.

Suppose we have characters $\lambda_{n+1}: \TSch_{n+1}(F_{v_0})\to \CC^\times$ and $\lambda_n : \TSch_n(F_{v_0}) \to \CC^\times$, chosen compatibly such that $\lambda_{n+1}|_{\TSch_{n}}=\lambda_n$. Our embeddings give us an inclusion of line bundles
\[
    \begin{tikzcd}
        \Lbundle_{\lambda_{n}}
            \arrow[r]
            \arrow[d]
            &
        \Lbundle_{\lambda_{n+1}}
            \arrow[d]
            \\
        \GS(U_{n})
            \arrow[r]
            &
        \GS(U_{n+1})
    \end{tikzcd}
\]
and in the reverse direction we get a corresponding map at the level of cohomology
\begin{align}
    \label{restr:map}
    H^*(\GS(\mathrm{U}_{n+1}), \Lbundle_{\lambda_{n+1}}) 
    \longrightarrow
    H^*(\GS(\mathrm{U}_{n}), \Lbundle_{\lambda_n}) ~.
\end{align}

We can then decompose this into a map of isotypic subspaces
\begin{align}
    \bigoplus_{\pi} H^*(\GS(\mathrm{U}_{n+1}),\Lbundle_{\lambda_{n+1}})(\pi^\infty)
    \longrightarrow
    \bigoplus_{\pi^\prime} H^*(\GS(\mathrm{U}_n),\Lbundle_{\lambda_n})((\pi^\prime)^\infty)~.
    \label{cohomology map isotypic decomposition}
\end{align}

Applying the Matsushima formula of \eqref{eq: cohomology decomposition} we get a map

\begin{align}
    \bigoplus_{\pi} H^*(\nn_{n+1},\pi_{v_0})_{-\lambda_{n+1}} \otimes \pi^\infty
    \longrightarrow
    \bigoplus_{\pi^\prime} H^*(\nn_{n},\pi^\prime_{v_0})_{-\lambda_n} \otimes (\pi^\prime)^\infty ~.
    \label{cohomology map after Matsushima}
\end{align}

\subsection{Context from n-cohomology.}

Now take $\rho_{n+1},\rho_n$ to be the half sums of positive roots and $\nn_{n+1}, \nn_n$ to be the negative root spaces for $\mathrm{U}_{n+1},\mathrm{U}_n$ respectively. From the earlier discussion of Borel subgroups, we had that $\nn_{n+1}$ restricts to $\nn_n$.
Suppose $\lambda_{n+1}+\rho_{n+1}$ and $\lambda_n+\rho_n$ are orthogonal to all roots for the respective root systems of $\mathrm{U}_{n+1}$ and $\mathrm{U}_{n}$. In view of \eqref{cohomology map after Matsushima}, the following is a necessary condition for \eqref{restr:map} to be nonzero: 

\begin{thm} \label{thm:nonvanish}
There exists $q \in \ZZ$ and TDLDS representations $\pi_{v_0}$ and $\pi^\prime_{v_0}$ of $\mathrm{U}_{n+1}(F_{v_0})$ and $\mathrm{U}_{n}(F_{v_0}),$ respectively, such that 
\[
    H^q(\nn_{n+1},\pi_{v_0})_{-\lambda_{n+1}} \neq 0 \text{ and }
    H^q(\nn_{n}, \pi^\prime_{v_0})_{-\lambda_n} \neq 0 ~.
\]
\end{thm}

\begin{proof}
    For $m=n,n+1$ let $\nn_m$ and $Q_m\subset X_m ...$ etc. be attached to $G=U_m$ as introduced in the previous sections. Note that $\dim_\CC Q_m$ does not depend on the choice of $Q_m$ as an orbit for the complexified maximal compact of the rank $m$ unitary group.
    By Corollary \ref{cor: easy to state non-vanishing}, there exist TDLDS representations $\pi_{v_0}$ and $\pi^\prime_{v_0}$ such that
    \[
        H^{\dim Q_{n+1}}(\nn_{n+1},\pi_{v_0})_{-\lambda_{n+1}}\neq 0
        \text{
        and
        }
        H^{\dim X_n - \dim Q_{n}}(\nn_{n},\pi^\prime_{v_0})_{-\lambda_{n}}\neq 0
    \]
    Now let $(a,b)$ be the signature of the rank $n=a+b$ Hermitian form defining $U_n$. Since $U_n$ and $U_{n+1}$ are quasi split, we can take $a=b$ or $b+1$, and assume $U_{n+1}$ has signature $(a,b+1)$. From formulas of \cite{mcgovern2012} and the fact that $a-b = (a-b)^2$ we get
    \begin{align*}
        \dim_\CC Q_{n+1} 
        &= \frac{a(a-1)+(b+1)b}{2} =ab
        \\
        &=
        \frac{n(n-1)}{2} - \frac{a(a-1)+b(b-1)}{2}
        = \dim_\CC X_{n}-\dim_\CC Q_{n} ~.
    \end{align*}
\end{proof}

\subsection{Compatibility with the Gan-Gross-Prasad conjecture}

The Gan-Gross-Prasad conjectures implies that there are global and local obstructions to the nonvanishing of \eqref{restr:map}. In this section, we will make this specific by relating $\eqref{restr:map}$ to the global period integral that occurs in the conjectures, and then discussing the local period integrals.

Let us write $[\mathrm{U}_n]=\mathrm{U}_n(F)\backslash \mathrm{U}_n(\A_F)$ where we recall $U_n$ is assumed anisotropic at the Archimedean places $v\neq v_0$. We first note the global obstruction in the following proposition.

\begin{prop}
    If \eqref{restr:map} is nonzero then there exist cuspidal automorphic TDLDS representations $\pi$ and $\pi^\prime$, and automorphic forms $\varphi \in \pi$ and $\varphi^\prime \in \pi^\prime$ such that
    \begin{align} \label{period}
    \int_{[\mathrm{U}_n]}\varphi(g)\varphi^\prime(g)dg \neq 0 ~.
    \end{align}
    Moreover, if \eqref{period} is nonzero then
    \begin{align} \label{central:value}
    L(\tfrac{1}{2},\widetilde{\pi} \otimes \widetilde{\pi}^\prime,r_{n+1} \otimes r_n) \neq 0 ~,
    \end{align}
    where we let $r_n:{}^L \mathrm{U}_n \lto \mathrm{Res}_{E/F}\GL_n$ denote the usual base change $L$-map \cite[\S 7]{GGP}
\end{prop}

\begin{proof}
    The second equation \eqref{central:value} is one direction of the Gan-Gross-Prasad conjectures which is a theorem for unitary groups. On the other hand, by our assumptions we have that $[\mathrm{U}_n]$ is compact. Let us consider a nonzero cohomology class
    \begin{align*}
        \sum_{i} \varphi_i\wedge \omega_i \in H^q(\GS(\mathrm{U_{n+1}}),\Lbundle_{\lambda_{n+1}}) ~.
    \end{align*}
    Here, we have used \ref{eq: cohomology decomposition} to obtain simple tensors $\varphi_i=\otimes_v \varphi_{i,v} \in \pi$ such that $\sum_{i} \varphi_{i,v_0}\wedge \omega_i \in H^q(\nn,\pi_{v_0})$ represents a nonzero $\nn$ cohomology class.

    By assumption, this class can be chosen such that its restriction to the smaller group
    \begin{align*}
        \sum_{i} (\varphi_i\wedge \omega_i)|_{\GS(\mathrm{U}_n)} \in H^q(\GS(\mathrm{U}_n),\Lbundle_{\lambda_n})
    \end{align*}
    is nonzero. Since $[U_n]$ is compact, we can apply Serre duality for complex compact analytic manifolds such that this restricted cohomology class pairs nontrivially with another nonzero cohomology class
    \begin{align*}
        \sum_{j} \varphi^\prime_j \wedge \omega^\prime_j \in  H^{\mathrm{\dim_{\CC} \GS(\mathrm{U_n})}-q}(\GS(\mathrm{U}_n),\Lbundle_{-\lambda_{n}-2\rho_n})
    \end{align*}
    constructed in the same manner using \ref{eq: cohomology decomposition}. Here we assumed $\lambda_n+\rho_n$ is fixed by the Weyl group, which implies $-\lambda_n-\rho_n=(-\lambda_n-2\rho_n)+\rho_n$ is also fixed by the Weyl group. Therefore, the Serre dual cohomology class obtained must also be TDLDS.
    
    We have thus chosen $\varphi=\varphi_i$ and $\varphi^\prime=\varphi^\prime_j$ such that for some compact open subgroup $K^\infty\leq \mathrm{U}_n(\A_F^\infty)$,
    \begin{align*}
        \int_{[U_n]} \varphi(g) \varphi^\prime(g) dg = \mathrm{Vol}(T(F_{v_0})\times K^{v_0})\int_{\GS(\mathrm{U}_n)^{K^\infty}} \varphi(g) \varphi^\prime(g) dg \neq 0
    \end{align*}
    where $K^{v_0}=\prod_{v_0\neq v | \infty} \mathrm{U}_n \times K^\infty$, so we are done.
\end{proof}

We now want to relate the period integral in the above proposition to local periods at the Archimedean places where the TDLDS occur.
Let $\pi^\prime_v$ and $\pi_v$ be the local unitary representations that occur in the factorization of the automorphic representations
\begin{align*}
    \pi^\prime=\otimes^\prime \pi^\prime_v
    \text{ and }
    \pi=\otimes^\prime \pi_v ~,
\end{align*}
and we use the same constructions as before such that $\pi^\prime_{v_0}$ and $\pi_{v_0}$ are TDLDS for $\mathrm{U}_n$ and $\mathrm{U}_{n+1}$, respectively. Note that since TDLDS are unitary representations, they come equipped with their respective inner products $(\cdot,\cdot)_{\pi^\prime_{v_0}}$ and $(\cdot,\cdot)_{\pi_{v_0}}$, which provide us with the matrix coefficients in the following proposition.

\begin{prop}
    If \eqref{period} is nonzero then there exist $\varphi_0\in \pi_{v_0}$ and $\varphi_0^\prime \in \pi^\prime_{v_0}$, such that the local period for TDLDS satisfies
    \begin{align} \label{local:period}
        \int_{\mathrm{U_n}(F_{v_0})} (\pi_{v_0}(g)\varphi_0,\varphi_0)_{\pi_{v_0}}(\pi^\prime_{v_0}(g)\varphi_0^\prime,\varphi_0^\prime)_{\pi^\prime_{v_0}} dg \neq 0~.
    \end{align}
\end{prop}

We note that the integral in \eqref{local:period} is convergent by \cite{harris2012}. We want to interpret this at the level of $\nn$ cohomology. Note that if we are given an intertwining map $\pi_{v_0} \to \pi^\prime_{v_0}$, then the inclusion $\nn_{n}\hookrightarrow \nn_{n+1}$ induced by the inclusion of groups $U_n \to U_{n+1}$ induces a restriction map at the level of $\nn$-cohomology
\begin{align} \label{n cohomology restriction map}
    H^*(\nn_{n+1}, \pi_{v_0}) \to H^*(\nn_n,\pi^\prime_{v_0}) ~.
\end{align}
We conclude by posing the following question:
\begin{question}
    Is the nonvanishing of the local period in \eqref{local:period} equivalent to the nonvanishing of the $\nn$ cohomology restriction map in \eqref{n cohomology restriction map}?
\end{question}

\section{A cohomological intertwining map for the TDLDS of special unitary groups of ranks 3 and 2}
\label{section: low rank GGP}

In this section, we provide evidence for the relationship between local GGP and $\nn$ cohomology for low rank unitary groups. However, we warn that the results given are only for special unitary groups.
Our goal is to prove that there exists an intertwining map from the TDLDS of $SU(2,1)$ considered in \cite{carayol1998} to a limit of discrete series of $SU(1,1)$ that induces a non-zero map on cohomology. Conversely, there is only one such representation of $SU(1,1)$ for which such a cohomological intertwining map can exist.

Specifically,
we take $\pi$ to be the Harish-Chandra module of the unique TDLDS of $SU(2,1)$. Take $\pi^\prime$ to be the Harish-Chandra module of any irreducible admissible representation of $SU(1,1)$. Take $\gC= \CC \otimes_{\RR} \mathrm{Lie}(SU(2,1))$ and $\gC^\prime = \CC \otimes_{\RR} \mathrm{Lie}(SU(1,1))$ to be the complexifications of the Lie algebra of each respective group. As before, we will consider certain negative root spaces $\nn\subset \gC$ and $\nn^\prime=\nn\cap \gC^\prime$. We will then consider a morphism of $\gC^\prime$ modules $J: \pi \to \pi^\prime$. If $J$ induces a non-vanishing restriction map of cohomology
\begin{align*}
    H^*(\nn,\pi) \to H^*(\nn^\prime , \pi^\prime).
\end{align*}
then we will show that $\pi^\prime$ must necessarily be an anti-holomorphic limit of discrete series for $SU(1,1)$. We will then explicitly construct a morphism $J$ that makes the cohomology restriction map non-zero.
This is the main result of the section, and is proven in Theorem \ref{Main Thm: Low rank GGP}.

\subsection{The construction of SU(2,1)}
Let 
\begin{align*}
    U_{Q}(\RR)=\{g\in \GL_3(\CC) | g^\dagger Q g Q^{-1} = 1 \}
\end{align*}
where $\dagger$ denotes conjugate transpose and $Q$ is the matrix
\[
    Q=\begin{pmatrix}
        1&&\\
        &1&\\
        &&-1\\
    \end{pmatrix}~.
\]
Let $G_\RR=SU_Q(\RR)$ be elements of $U_Q(\RR)$ with determinant 1. Since $U_Q(\RR)\subset \GL_3(\CC)$, we identify the complexification $\CC\otimes_\RR \mathrm{Lie}(U_Q(\RR))=M_3(\CC)$ with the $3\times 3$  complex matrices. We can further identify $\gC = \mathrm{Lie}(G_\RR)\otimes_\RR \CC$ with the trace zero elements of $M_3(\CC)$. Using this identification, we take a maximal compact subgroup $K_\RR=S(U(2)\times U(1))$, such that the complexified Lie algebra $\kC:= Lie(K_\RR)\otimes_\RR \CC$ is given inside $M_3(\CC)$ as the trace zero block diagonal matrices
\begin{align*}
    \kC = \begin{psmatrix}
        *&*&\\
        *&*&\\
        &&*
    \end{psmatrix} \cap \gC.
\end{align*}
Furthermore, we see that the trace zero diagonal elements
\begin{align*}
    \tC = \begin{psmatrix}
        *&&\\
        &*&\\
        &&*
    \end{psmatrix}\cap \gC
\end{align*}
give the complexification of a compact maximal torus. Let $\Phi=\Phi(\gC,\tC)$ be the resulting root system. We define the following matrices
\begin{align*}
    X_{-\alpha} = \begin{psmatrix}
        0&0&0\\
        0&0&0\\
        0&1&0
    \end{psmatrix}
    ,~
    X_{-\beta} = \begin{psmatrix}
        0&0&1\\
        0&0&0\\
        0&0&0
    \end{psmatrix}
    ,~ \text{ and }
    X_{-\alpha-\beta} = \begin{psmatrix}
        0&1&0\\
        0&0&0\\
        0&0&0
    \end{psmatrix}.
\end{align*}
We then let $\Phi^+=\{\alpha,\beta,\alpha+\beta\}\subset \Phi$ be the resulting system of positive roots relative to $\tC$. The negative root spaces are
\begin{align*}
    \gC_{-\alpha} = \CC. X_{-\alpha}
    ,~ 
    \gC_{-\beta} = \CC. X_{-\beta}
    ,~ \text{ and }
    \gC_{-\alpha-\beta} = \CC. X_{-\alpha-\beta}.
\end{align*}
Let $X_{\alpha}, X_{\beta}, X_{\alpha+\beta}$ respectively denote the transpose matrices, generating positive root spaces $\gC_{\alpha}, \gC_{\beta}, \gC_{\alpha+\beta}$.
From the description of $\kC$, note that $\alpha,\beta$ give the noncompact simple roots, and $\alpha+\beta$ are compact.

\subsection{The embedding of SU(1,1) inside SU(2,1)}
We will provide a construction of $G^\prime_\RR = SU(1,1)$ and its root systems such that $\gC_{\alpha}$ restricts nontrivially, and the half sum of positive roots for the larger group restricts to the half sum of positive roots for the smaller group.

Let $Q^\prime=\begin{psmatrix}
    1&\\
    &-1
\end{psmatrix}$ and define
\begin{align*}
    U_{Q^\prime}(\RR) = \{g\in \GL_2(\CC) | g^\dagger Q^\prime g (Q^\prime)^{-1} = 1\}.
\end{align*}
Write $U(1,1) = U_{Q^\prime}(\RR)$ and define $G^\prime_\RR=SU(1,1)$ to be the determinant $1$ elements of $U(1,1)$. Take the corresponding complexified Lie algebra $\gC^\prime = \CC \otimes_{\RR} \mathrm{Lie}(G^\prime_\RR)$. Once again we identify $\gC^\prime$ with the trace zero elements of $M_2(\CC)$.
We then take $\tC^\prime\subset \gC^\prime$ to be the Cartan subalgebra corresponding to the trace zero diagonal elements. As before, $\tC^\prime$ is the complexification of a compact maximal torus. The compact maximal torus equals the maximal compact subgroup 
$K^\prime_\RR = \{\begin{psmatrix}
    e^{i\theta} & \\
    & e^{-i\theta}
\end{psmatrix} : \theta \in \RR\}$.

Let $X_{-\alpha^\prime}= \begin{psmatrix}
    0 & 0\\
    1 & 0
\end{psmatrix}$ and $X_{\alpha^\prime}= \begin{psmatrix}
    0 & 1\\
    0 & 0
\end{psmatrix}$. 
We then obtain a positive root system $\Phi^{\prime+}=\{\alpha^\prime\}$ for $\tC^\prime$. The unique positive root $\alpha^\prime$ is given by the negative root space $\gC^\prime_{-\alpha^\prime}=\CC.X_{-\alpha^\prime} = \begin{psmatrix}
    0 & 0\\
    * & 0
\end{psmatrix}$ and positive root space $\gC^\prime_{\alpha^\prime} = \CC.X_{\alpha^\prime}=\begin{psmatrix}
    0 & *\\
    0 & 0
\end{psmatrix}$. 
Let $\rho^\prime=(1/2)\alpha^\prime$ be the half sum of positive roots. Furthermore, we embed $G^\prime_\RR$ inside $G_\RR$ as the block diagonal
\[
    SU_2=\begin{psmatrix}
        1 &\\
        & \SL_2(\CC)\\
    \end{psmatrix}
    \cap
    SU_3~.
\]

Under the embedding above, we obtain an inclusion $\gC^\prime \subset \gC$ induced by the inclusion of the block diagonal, as follows:
\[
\begin{tikzcd}
    \gC^\prime
    \arrow[r, phantom, "\subset"]
    \arrow[d, phantom, sloped, "\subset"]
    &
    \gC
        \arrow[d, phantom, sloped, "\subset"]
    \\
    M_2(\CC)
        \arrow[r, phantom, "\subset"]
    &
    M_3(\CC)
\end{tikzcd}
\]
It follows that 
\begin{align}
    \label{eq: restriction of subspaces}
    \tC \cap \gC^\prime &= \tC^\prime, 
    \\ \notag
    \gC_\alpha \cap \gC^\prime &= \gC^\prime_{\alpha^\prime}, 
    \\ \notag
    \gC_{-\beta}\cap \gC^\prime &= 0, \text{ and }
    \\ \notag
    \gC_{-\alpha-\beta}\cap \gC^\prime &= 0.
\end{align}

We now record how $\alpha$ and $\beta$ restrict to $SU_2$ in the following proposition:
\begin{prop}
    \label{prop: restriction of weights}
    Given the roots $\alpha$ and $\beta$ as specified above, we have
    \begin{align*}
        \alpha|_{\tC^\prime} = \alpha^\prime, \text{ and }
        \beta|_{\tC^\prime}  = -\frac{1}{2}\alpha^\prime.
    \end{align*}
\end{prop}
\begin{proof}
    Note that that $\alpha,\beta$ and $\alpha^\prime$ are given by
    \begin{align*}
        \alpha \left(\begin{psmatrix}
            a&&\\
            &b&\\
            &&c
        \end{psmatrix}
        \right)
        &=b-c = \alpha^\prime(\begin{psmatrix}
            b&\\
            &c
        \end{psmatrix})
        \text{ and}
        \\
        \beta \left(\begin{psmatrix}
            a&&\\
            &b&\\
            &&c
        \end{psmatrix}
        \right)
        &=c-a \text{ for } a,b,c\in \CC
    \end{align*}
    
    Fix a basis element $H$ of $\tC^\prime$ by
    \begin{align*}
        H = \begin{psmatrix}
            1 & \\
            & -1
        \end{psmatrix}
        = \begin{psmatrix}
            0&&\\
            &1&\\
            &&-1
        \end{psmatrix}~.
    \end{align*}
    Plugging in gives $\alpha(H)=2, \beta(H)=-1$ and $\alpha^\prime(H)=2$ so we are done.
\end{proof}

\subsection{Cocycle for SU(2,1)}

Now let $\pi$ be the unique TDLDS of $G(\RR)$ and let $\pi^\prime$ be any irreducible admissible $(\gC^\prime, K^\prime)$ module of $G^\prime(\RR)$.
Suppose $J:\pi\to \pi^\prime$ is a morphism of $\gC^\prime$ modules. The inclusion $\nn^\prime \subset \nn$ and $J$ both determine a restriction map of cohomology
\begin{align}
    \label{eq: cohomology restriction SU(2,1) to SU(1,1)}
    r: H^*(\nn,\pi)_\rho \to H^*(\nn^\prime,\pi^\prime)_{\rho^\prime}.
\end{align}
With respect to these choices, we record the following fact which reduces all cohomological questions to the question of intertwining maps:
\begin{lem} \label{lem: nonzero cohomology restriction implies LDS}
    The restriction map in $\eqref{eq: cohomology restriction SU(2,1) to SU(1,1)}$ is nonzero if and only if $\pi^\prime$ is the anti-holomorphic limit of discrete series whose one-dimensional minimal $K$ type is generated by $J(X_\beta \phi_0)$.
\end{lem}
\begin{proof}
    Carayol constructs a nontrivial cocycle that generates $H^1(\nn, \pi)$ as follows. Let $\omega^{-\alpha}$, $\omega^{-\beta}$, $\omega^{-\alpha-\beta} \in \nn^\vee$ be a dual basis of $X_{-\alpha},X_{-\beta}, X_{-\alpha-\beta}$. Let $\phi_0$ be an element of the one-dimensional minimal $K$ type of $\pi$. There exist nonzero scalars $c_\alpha, c_\beta \in \CC^\times$, determined uniquely by $\phi_0$, such that
    \begin{align}
        \eta:=\phi_0 \otimes \omega^{-\alpha-\beta} + c_{\alpha}X_{\alpha} \phi_0 \otimes \omega^{-\beta} + c_{\beta}X_\beta \phi_0 \otimes \omega^{-\alpha} \in \pi\otimes \wedge^1 \nn^\vee
    \end{align}
    is a cocycle in the Chevalley Eilenberg complex $\pi\otimes \wedge^\bullet \nn^\vee$. This cocycle represents the unique nontrivial class of $H^1(\nn,\pi)$ up to scaling. Let $\psi:=c_{\beta}J(X_{\beta} \phi_0)$.
    Note that \eqref{eq: restriction of subspaces} implies that the image of $\eta$ under $r$ is equal to 
    \[
        r(\eta) = \psi \otimes \omega^{-\alpha^\prime} .
    \]
    Here we have set $\omega^{-\alpha^\prime}$ to be the element dual to $X_{-\alpha^\prime}$ in $(\nn^\prime)^\vee$ which has dimension one.
    Note that $\psi$ is a weight vector, since it is the image of a weight vector of weight $\beta$ under the $\gC^\prime$ module map $J$ and $G^\prime(\RR)$ is connected. By \ref{prop: restriction of weights}, the $\tC^\prime$ weight for $\psi$ is equal to 
    \[\beta|_{\tC^\prime} = -\rho^\prime.
    \]
    We now want to use this description of $r(\eta)$ to prove the if and only if statement of the Lemma. 
    
    We first prove the converse $(\impliedby)$. Since $SU(1,1)$ has no compact roots, we can apply \cite{williams1988}. Therefore, $H^1(\nn^\prime, \pi^\prime)_{\rho^\prime}$ has a nontrivial cocycle $\psi^\prime\otimes \omega^{-\alpha^\prime}$ given by tensoring any nonzero element $\psi^\prime$ of the unique one-dimensional minimal $K^\prime$ type of $\pi^\prime$ with $\omega^{-\alpha^\prime}$. Our assumption on $\pi^\prime$ implies the minimal $K^\prime$ type equals the weight space of $-\rho^\prime$, and therefore contains $\psi$. Hence, we can set $\psi=\psi^\prime$ by applying a non-zero scalar, and it follows that $r(\eta)=\psi^\prime\otimes \omega^{-\alpha^\prime}$ is the cocycle of a nontrivial cohomology class.
    
    It remains to prove the forward implication $(\implies)$. Suppose that $r(\eta)=\psi\otimes \omega^{-\alpha^\prime} \in \bigwedge^1 (\nn^\prime)^\vee\otimes \pi^\prime$ is a cocycle representing a non-trivial cohomology class. We must determine what is $\pi^\prime$.

    The classification of irreducible admissible $(\gC^\prime,K^\prime)$ modules provides additional information about the $\gC^\prime$ action on $\psi$. In particular, $\psi$ is a $\tC^\prime$ weight vector for $-\rho^\prime$ so it must be a minimal $K^\prime$ type. The only possible representations with such a minimal $K^\prime$ type are the principal series or the limits of discrete series. Therefore, there exists $s\in \CC$ such that 
    \begin{align}
        (s-\frac{1}{2})^2 \psi = X_{-\alpha^\prime} X_{\alpha^\prime}\psi.
    \end{align}

    However, it must be the case that $X_{-\alpha^\prime}X_{\alpha^\prime}\psi=0$. Indeed, if we suppose otherwise then $r(\eta) \in \bigwedge^1 \nn^\prime \otimes \pi^\prime$ would be the coboundary of $(s-\frac{1}{2})^{-2}X_{\alpha^\prime}\psi \in \bigwedge^0 \nn^\prime \otimes \pi^\prime$. This contradicts the assumption that $r(\eta)$ is a non-trivial cocycle.
    
    We therefore see that $s=1/2$. The parameter $s$ determines the eigenvalue of the Casimir operator $(1/4)(H^2+X_{-\alpha}X_\alpha+X_\alpha X_{-\alpha})$ on $\pi^\prime$ as $s(1-s)=1/4$.
    The only such irreducible admissible representation with Casimir eigenvalue $1/4$ and non-zero vector of weight $-\rho^\prime$ is the anti-holomorphic limit of discrete series. This completes the proof of the lemma.
\end{proof}

\subsection{The Harish-Chandra module for SU(2,1)}

We shall now write the $K$ type decomposition of the TDLDS $\pi$, and recall the action of $\gC$. We follow the notation established in \cite{carayol1998} based on \cite{jw1977}. There, it is established that $\pi$ admits a decomposition into $K$ types
\begin{align}
    \pi = \bigoplus_{p,q\geq 0}\mathcal{H}^{p,q}
\end{align}
where each $\mathcal{H}^{p,q}$ is an irreducible representation of $K$. Every such $\mathcal{H}^{p,q}$ is given a basis $\{e_{p,q}(j)\}_{p\geq j \geq -q}$. On this basis, we recall the action of $\gC$ in the following theorem:
\begin{thm}[\cite{carayol1998},\cite{jw1977}]
    \label{thm: action of root vectors}
    The action of the negative root vectors is given by
    \begin{align*}
        X_{-\alpha-\beta}.e_{p,q}(j) &= -(p-j)e_{p,q}(j+1)
        \\
        X_{-\beta}.e_{p,q}(j) &= \frac{1}{p+q+1}((q+1)^2 e_{p,q+1}(j)-p(p-j) e_{p-1,q}(j)))
        \\
        X_{-\alpha}. e_{p,q}(j) &= \frac{1}{p+q+1} ((p+1)^2 e_{p+1,q}(j+1) + q(p-j) e_{p,q-1}(j+1)).
    \end{align*}
    Furthermore, the action of the positive root vectors is given as
    \begin{align*}
        X_{\alpha+\beta}.e_{p,q}(j)
        &=
        -(q+j).e_{p,q}(j-1)
        \\
        X_{\beta}.e_{p,q}(j)
        &=
        \frac{1}{p+q+1}((p+1)^2 e_{p+1,q}(j)-q(q+j)e_{p,q-1}(j))
        \\
        X_{\alpha}.e_{p,q}(j)
        &=-\frac{1}{p+q+1}((q+1)^2 e_{p,q+1}(j-1)+ p(q+j) e_{p-1,q}(j-1))
    \end{align*}
\end{thm}
\begin{proof}
    The first 5 identities are quoted verbatim from \cite{carayol1998} and the last is obtained by expanding from $X_\alpha = [X_{\alpha+\beta}, X_{-\beta}]$
\end{proof}

For each weight $\mu \in \Lambda$ in the integral weight lattice $\Lambda \subset \tC^\vee$, let $\pi_\mu$ denote the $\mu$ weight space of $\pi$. 
Note in this case that $e_{0,0}(0)$ is the minimal $K$ type vector since it is annihilated by both $X_{\alpha+\beta}$ and $X_{-\alpha-\beta}$. From there, we record the following:
\begin{prop}
    \label{prop: weight basis by K types}
    Each basis vector in $\{e_{p,q}(j)\}_{p\geq j \geq -q}^{p,q \geq 0}$ of $\pi$ satisfies
    \begin{align}
        \label{eq: K types to weight vectors}
        e_{p,q}(j)\in \pi_\mu \text{ where }
        \mu = -j \alpha + (p-q-j) \beta.
    \end{align}
    Therefore, if $\mu = u\alpha + v\beta$ then
    \begin{align}
        \label{eq: weight vectors to K types}
        \pi_\mu = 
            \bigoplus_{k\geq 0} \CC . e_{(v^+ -u^-) +k,(u^+- v^-) +k}(-u)
    \end{align}
    where for any integer $x$ we define
    \begin{align*}
        x^+ 
        =
        \begin{cases}
            x \text{ if } x\geq 0
            \\
            0 \text{ if } x\leq 0
        \end{cases}
        \text{ and } x^- 
        &=
        \begin{cases}
            x \text{ if } x \leq 0
            \\
            0 \text{ if } x \geq 0
        \end{cases}
    \end{align*}
\end{prop}
\begin{proof}
    We will first argue that \eqref{eq: weight vectors to K types} follows from \eqref{eq: K types to weight vectors}. By assumption, $e_{p,q}(j)\in \pi_\mu$ if and only if 
    \begin{align}
        \label{eq: pqj from mu}
        p\geq j \geq -q \text{ and } -j\alpha+(p-q-j)\beta = \mu
    \end{align}
    both hold. If $\mu=u\alpha+v\beta$ is fixed, then $j=-u$ is determined by $\mu$ and $q=p-(v-u)$ is determined by $p$. Furthermore, we assumed $p,q\geq 0$, so we have
    \begin{align*}
        p \text{ is greater than } 0, v, -u, v-u
    \end{align*}
    Let $p^\prime$ be the smallest integer satisfying this condition. We get $p=p^\prime+k$ for $k\geq 0$, and $p^\prime = \max\{0,v,-u,v-u\}=v^+-u^-$. Similarly we can show that the smallest $q^\prime$ that can be plugged into $q$ is $q^\prime=u^+-v^-$ and we readily check that $q^\prime = p^\prime -(v-u)$. We obtained 
    \[
        e_{p,q}(j) \in \mu \iff p=v^+-u^- + k \text{ and } q= u^+-v^-+k \text{ for some } k\in \ZZ_{\geq 0},
    \]
    and this shows \eqref{eq: weight vectors to K types} follows from \eqref{eq: K types to weight vectors}.

    It remains to show \eqref{eq: K types to weight vectors}. We need to show $e_{p,q}(j)$ is a weight vector for $-j\alpha +(p-q-j)\beta$. 
    Notice that
    \begin{align*}
        X_{\beta}^p. e_{0,0}(0) = \frac{(p!)^2}{p!} e_{p,0}(0)
    \end{align*}
    due to a cancellation from plugging in $q=0$ in Theorem \ref{thm: action of root vectors}. Similarly, we get
    \begin{align}
        \label{eq: highest weight vector}
        X_{\alpha}^q X_\beta^p e_{0,0}(0)
        =
        \frac{(-1)^q(p!q!)^2}{(p+q)!} e_{p,q}(-q)
    \end{align}
    due to a cancellation from plugging in $q+j=0$ in Theorem \ref{thm: action of root vectors}. Finally, if $p\geq j \geq -q$ then $j^\prime:=j+q$ gives us $p+q\geq j^\prime \geq 0$. We then calculate
    \begin{align*}
        X_{-\alpha-\beta}^{j^\prime}X_{\alpha}^q X_\beta^p e_{0,0}(0)
        =\lambda_0
        e_{p,q}(-q+j^\prime)
        \\
        \text{ where }
        \lambda_0=\frac{(-1)^{j^\prime+q}(p!q!)^2}{(p+q-j^\prime)!}.
    \end{align*}
    Since we took $p,q\geq 0$ and $p\geq j \geq -q$, it follows that $\lambda_0\neq 0$. Therefore, $e_{p,q}(j)$ is in the $\mu= q\alpha+p\beta-j^\prime(\alpha+\beta) = -j \alpha + (p-q-j)\beta$ weight space $\pi_\mu$ and we are done.

\end{proof}

We will now rewrite a basis for $\pi_{\alpha-k\beta}$ for any integer $k\in \ZZ$. This basis will simplify the construction of the intertwining map. It is recorded in the following lemma as a consequence of the formulas in Theorem \ref{thm: action of root vectors}.

\begin{lem}
    \label{lem: GGP root string}
    Suppose $\phi$ is a $\Phi_K^+$ highest $K$ weight vector generating a $K$ type in $\pi$ and furthermore assume $\phi$ has weight $\mu$. Then, 
    \begin{enumerate}
        \item \label{sublem: GGP root string easier basis}
            the corresponding weight space $\pi_{\mu}$ admits a basis given by
            \[
                \{(X_{-\alpha} X_{\alpha})^k . \phi \}_{k\geq 0}
            \]
        \item \label{sublem: GGP root string PBW basis}
            By Poincare-Birkhoff-Witt, we also have a basis of $\pi_\mu$ given by
            \[
                \{X_{-\alpha}^{k} X_{\alpha}^{k} . \phi\}_{k\geq 0}.
            \]
            Furthermore, the complex spans of the subset of these bases at index $k\geq 1$ are identical:
            \begin{align}\label{eq: span from 1}
                \mathrm{Span} \{(X_{-\alpha} X_\alpha)^k\}_{k\geq 1}
                =\mathrm{Span} \{X_{-\alpha}^k X_{\alpha}^k\}_{k\geq 1}.
            \end{align}
        \item \label{sublem: GGP root string isomorphism}
            Now assume $\mu=p\beta$ is a non-negative integer multiple of $\beta$. Then, for all non-negative integers $r\in \ZZ_{\geq 0}$, the root vector $X_{-\alpha}$ acts by isomorphisms on the ``$-\alpha$ root string"
            \[
                X_{-\alpha}: \pi_{\mu-r\alpha} 
                \overset{\sim}{\longrightarrow}
                \pi_{\mu-(r+1)\alpha}
            \]
    \end{enumerate}

\end{lem}
\begin{proof}

    We begin with 
    \ref{lem: GGP root string}
    \eqref{sublem: GGP root string easier basis}.
    
    If we suppose $e_{p,q}(j_0)$ is a highest weight vector, then it must be annihilated by $X_{\alpha+\beta}$. Applying Theorem \ref{thm: action of root vectors}, we get $0=X_{\alpha+\beta}. e_{p,q}(j_0) = (q-j_0).e_{p,q}(j_0-1)$ so $j_0=-q$. Hence, any highest weight vector $\phi$ is a scalar multiple of $e_{p,q}(-q)$ for non-negative integers $p,q\geq 0$ that we now fix. Let us write $e_k:=e_{p+k,q+k}(-q)$ and $V_k = \CC e_k$. By Proposition \ref{prop: weight basis by K types}, $\phi$ is a weight vector for $\mu = q\alpha + p\beta $, and 
    \[
        \pi_{\mu} = \bigoplus_{k\geq 0} V_k
    \]
    We want to show $\pi_\mu$ has basis $\{(X_{-\alpha} X_{\alpha})^k. e_0\}_{k\geq 0}$.
    We claim by induction on $N\geq 0$ that
    \begin{align}
        \label{eq: induction for basis}
        \bigoplus_{0\leq k \leq N} V_k = \mathrm{Span}_\CC \{(X_{-\alpha} X_{\alpha})^k . e_0\}_{0\leq k \leq N}
    \end{align}
    When $N=0$ this is clearly true, so assume \eqref{eq: induction for basis} holds for $N$. By induction hypothesis, $e_N:=e_{p+N,q+N}(-q) \in V_N$ can be written as an element of the right hand side of \eqref{eq: induction for basis}. It then follows that $X_{-\alpha}X_\alpha e_N$ is spanned by $\{ (X_{-\alpha} X_\alpha)^k. e_0\}_{1\leq k\leq N+1}$.
    
    We now claim there exists some $\lambda_0\neq 0$ such that
    \begin{align*}
        X_{-\alpha} X_{\alpha} e_N = \lambda_0 e_{N+1}+w,
    \end{align*}
    where $w \in \bigoplus_{0\leq k \leq N} V_k$ is a linear combination of $e_1,...,e_N$.
    
    Note that by \ref{thm: action of root vectors} applied to $X_{-\alpha}X_{\alpha}$, we obtain scalars $\lambda_0,\lambda_1, \lambda_2 \in \CC$ such that
    \begin{align*}
        X_{-\alpha} X_{\alpha}. e_N
        &=
        \lambda_0 e_{N+1} + \lambda_1 e_N + \lambda_2 e_{N-1}.
    \end{align*}
    The expression above will be well defined due to cancellations in $\lambda_1$ and $\lambda_2$ that will result when $N=0$. Furthermore, the calculation from 
    \ref{thm: action of root vectors} will result in the following expression for $\lambda_0$:
    \[
        \lambda_0 = \frac{(p+N+1)^2(q+N+1)^2}{(p+q+2N+1)(p+q+2N+2)} \neq 0.
    \]
    Hence, $\lambda_0$ will be a nonzero scalar since $p,q,N\geq 0$. This completes the claim.

    Let $v^\prime:= X_{-\alpha} X_\alpha.e_N$ and $V^\prime_{N^\prime}= \mathrm{Span}_\CC \{(X_{-\alpha}X_\alpha)^k. e_0\}_{0\leq k \leq N^\prime}$ for any $N^\prime \geq 0$. We have just shown that $v^\prime \not\in V^\prime_N$ and 
    \begin{align}
        \label{eq: easy basis induction step vector subspace inclusions}
        V^\prime_{N+1} \supset \CC.v^\prime + V^\prime_N \subset \bigoplus_{0\leq k \leq N+1} V_k
    \end{align}
    Hence, $V^\prime_{N+1}$ and $\bigoplus_{0\leq k \leq N+1} V_k$ are vector spaces of dimension at most $N+1$, which contain the same $N+1$ dimensional subspace $\CC.v^\prime + V^\prime_N$, so all of the inclusions in \eqref{eq: easy basis induction step vector subspace inclusions} are equalities.
    This finishes the induction and we have shown that \eqref{eq: induction for basis} holds for all $N\geq 0$. 
    Therefore, $\{(X_{-\alpha}X_\alpha)^k.e_0\}_{k \geq 0}$ is linearly independent and spans all of $\pi_\mu$, 
    so we are done with
    \ref{lem: GGP root string}
    \eqref{sublem: GGP root string easier basis}.

    We now prove 
    \ref{lem: GGP root string}
    \eqref{sublem: GGP root string PBW basis}, by rehashing the Poincare-Birkhoff-Witt theorem. Write $H_{\alpha}=[X_\alpha,X_{-\alpha}]$.
    By applying 
    \[
        X_\alpha X_{-\alpha} = [X_{\alpha}, X_{-\alpha}]+X_{-\alpha} X_\alpha
        = H_\alpha + X_{-\alpha} X_\alpha
    \]
    we obtain
    \begin{align*}
        &X_{-\alpha} X_{\alpha} X_{-\alpha}^N X_{\alpha}^N. e_0
        =
        X_{-\alpha}^{N+1} X_{\alpha}^{N+1}. e_0 +
        \sum_{t=1}^N X_{-\alpha}^t H_\alpha X_{-\alpha}^{N-t} X_{\alpha}^N. e_0.
    \end{align*}
    Here we have formally defined $\sum_{t=1}^N$ to be zero if $N=0$.
    Note that $e_0$ is a weight vector for $\mu$, and $X_{\pm\alpha}$ sends weight vectors for $\mu^\prime$ to weight vectors for $\mu^\prime \pm \alpha$. Letting $c$ be the scalar $c=\sum_{t=1}^N(\mu+t\alpha)(H_\alpha))$, we obtain
    \begin{align}
        \label{eq: PBW rearrangement}
        X_{-\alpha} X_{\alpha} X_{-\alpha}^N X_{\alpha}^N e_0
        =&
        X_{-\alpha}^{N+1} X_{\alpha}^{N+1} e_0 +
        c X_{-\alpha}^{N} X_{\alpha}^N e_0.
    \end{align}
    
    Using the above identity, we can argue by induction on $N$ that $\{X_\alpha^k X_{-\alpha}^k e_0\}_{k=0}^N$ is linearly independent and generates the same vector space as $\{(X_\alpha X_{-\alpha})^k e_0\}_{k=0}^N$ for all $N$. This argument, which we omit, implies that $\{X_{\alpha}^k X_{\alpha}^k. e_0\}_{k\geq 0}$ is a basis of $\pi_\mu$. This proves the first half of 
    \ref{lem: GGP root string}
    \eqref{sublem: GGP root string PBW basis}
    and it remains to prove \eqref{eq: span from 1}.
    
    We claim that each $X_{-\alpha}^N X_\alpha^N. e_0$ is a linear combination of $X_{-\alpha}X_\alpha e_0,..., (X_{-\alpha}X_\alpha)^N e_0$ for all $N\geq 1$. Indeed, this is clearly true for $N=1$. Assume by induction hypothesis that
    \[
        X_{-\alpha}^N X_\alpha^N. e_0 = \sum_{i=1}^N \lambda_i (X_{-\alpha}X_\alpha)^k.e_0.
    \]
    Applying this to \eqref{eq: PBW rearrangement}, we obtain
    \begin{align*}
        X_{-\alpha}^{N+1} X_\alpha^{N+1}. e_0
        &=(X_{-\alpha}X_\alpha-c) X_{-\alpha}^N X_\alpha^N. e_0
        \\
        &=\sum_{i=1}^N \lambda_i (X_{-\alpha}X_\alpha)^{k+1}.e_0
        +c\sum_{i=1}^N \lambda_i (X_{-\alpha}X_\alpha)^k.e_0.
    \end{align*}
    This expression is a linear combination of the $(X_{-\alpha} X_\alpha)^k.e_0$ for $k\geq 1$, which completes the induction. This finishes \ref{lem: GGP root string}
    \eqref{sublem: GGP root string PBW basis}.

    Finally, we need to show
    \ref{lem: GGP root string}
    \eqref{sublem: GGP root string isomorphism}.
    Write once again $\mu = q\alpha+p\beta$, which we note is a sum of simple roots. We have assumed $\mu$ is a non-negative integer multiple of $\beta$ and therefore $q=0$. If $r$ is a non-negative integer, then define
    \begin{align*}
        e^r_k = e_{p+r+k,k}(r)
        \text{ and }
        V^r_k = \CC e^r_k.
    \end{align*}
    By Proposition \ref{prop: weight basis by K types} we can decompose the $\mu-r\alpha$ weight space as
    \[
        \pi_{\mu-r\alpha} = \bigoplus_{k\geq 0} V^r_k
    \]

    Let us define $A_k^i \in \CC$ for non-negative integers $i,k\geq 0$ by
    \[
        A_k^i = \begin{cases}
            (p+r+k+1)^2/(p+r+2k+1) & \text{ if } i=k
            \\
            k(p+k)/(p+r+2k+1) & \text{ if } i =k-1
            \\
            0 & \text{ otherwise.}
        \end{cases}
    \]
    These are precisely the coefficients that occur in Theorem \ref{thm: action of root vectors}, such that
    \begin{align*}
        X_{-\alpha} e^r_k
        &=X_{-\alpha} e_{p+r+k,k}(r)
        \\
        &=\frac{(p+r+k+1)^2}{p+r+2k+1} e_{p+r+1+k,k}(r+1) 
        + \frac{k(p+k)}{p+r+2k+1} e_{p+r+k,k-1}(r+1)
        \\
        &=: A_k^k e_k^{r+1} +A_k^{k-1} e_{k-1}^{r+1}
    \end{align*}
    This expression is well defined, ie. there is no need to worry about $e_{-1}^{r+1}$, because $A_k^{k-1}=0$ when $k=0$. Notice that $A_k^k$ is always nonzero because $p,q,k\geq 0$. Therefore, the matrix $A:=\{A^i_{k}\}_{i,k\geq 0}$ is invertible, since it is upper triangular with non-zero diagonal entries, and entries equal to zero above the super diagonal. In particular, each row and column of $A$ will have only finitely many non-zero entries. The inverse can therefore be produced inductively using Gaussian elimination, and will not involve infinite column vectors. The matrix $A$ is precisely the one representing the linear transformation
    \[
        X_{-\alpha}: \pi_{\mu-r\alpha} \longrightarrow\pi_{\mu-(r+1)\alpha}
    \]
    with respect to the bases $\{e_k^r\}\subset \pi_{\mu-r\alpha}$ and $\{e_k^{r+1}\} \subset \pi_{\mu-(r+1)\alpha}$. Therefore we have shown that $X_{-\alpha}: \pi_{\mu-r\alpha} \to \pi_{\mu-(r+1)\alpha}$ is an isomorphism of vector spaces, finishing the proof.
    
\end{proof}

\subsection{The Harish-Chandra module for SU(1,1)}
We recall the $(\gC^\prime,K^\prime)$ module of the anti-holomorphic limit of discrete series for $G^\prime = SU(1,1)$. We shall write
\[
    \pi^\prime = \bigoplus_{r\geq 0} \pi^\prime_r 
    \text{ where }
    \pi^\prime_r=\pi^\prime_{-(r+\frac{1}{2})\alpha^\prime}
\]
Here, each $\pi^\prime_r$ is the one-dimensional weight space for $-(r+1/2)\alpha^\prime$, and $\pi^\prime_0$ is the minimal $K^\prime$ type. Let $\psi_0\in \pi^\prime_0$ denote any choice of non-zero minimal $K^\prime$ type vector.

Recall from our construction of the root spaces that $X_{\pm \alpha} \in \gC$ and $X_{\pm \alpha^\prime}\in \gC^\prime$ coincide under the inclusion induced by $G^\prime \hookrightarrow G$. We will therefore write
\begin{align*}
    X &= X_{-\alpha} = X_{-\alpha^\prime},
    \\
    Y &= X_{\alpha} = X_{\alpha^\prime}, \text{ and }
    \\
    H&=[X,Y]
\end{align*}
This is consistent with the notation that was previously used for $H=H_\alpha$.
We then recall the standard fact that $X^r \psi \in \pi^\prime_r$ will be a non-zero element that spans $\pi^\prime_r$.

\subsection{Intertwining map}
We want to construct a morphism $J: \pi_3 \to \pi_2$ that is compatible with the $\gC^\prime$ action. 
Setting $\phi_0 = e_{0,0}(0)$, recall that $\phi_0$ spans the minimal $K$ type of $\pi_3$. Let $\phi_\beta=X_\beta. \phi_0$.
Recall from Lemma 
\ref{lem: GGP root string} 
parts \eqref{sublem: GGP root string easier basis} 
and   \eqref{sublem: GGP root string isomorphism}
that $\{(XY)^k\phi_\beta\}_{k\geq 0}$ is a basis for $\pi_\beta$, and that $X^r: \pi_\beta \to \pi_{\beta-r\alpha}$ is an isomorphism for all non-negative integers $r\geq 0$.

We now want to construct a linear transformation $J:\pi\to \pi^\prime$. We will proceed by defining $J$ on each weight space $\pi_\mu$ for every $\mu$. The compatibility with the $\gC^\prime$ action will then be checked afterwards.

We start by declaring that if a weight $\mu$ cannot be written in the form of $\beta - r\alpha$ for any $r\geq 0$, then
\begin{align*}
    J(\pi_\mu) = 0 \text{ for } \mu \not\in \beta-\ZZ_{\geq 0}\alpha
\end{align*}
We then need to define $J$ on the relevant weights $\beta-r\alpha$ for all $r\geq 0$. If $r=0$, then we do so by defining $J$ on each basis element $(XY)^k.\phi_\beta\in \pi_\beta$ by
\begin{align*}
    J((XY)^k.\phi_\beta) = \begin{cases}
        \psi &\text{if } k=0
        \\
        0 &\text{otherwise}
    \end{cases} 
\end{align*}
It remains to define $J$ when $r\geq 1$, which we do so inductively. Any element of $\pi_{\beta-r\alpha}$ can be written as $X\phi$ for $\phi\in \pi_{\beta-(r-1)\alpha}$. Having defined $J(\phi)$ by way of induction, we can then define
\begin{align}
    \label{eq: inductive definition of J}
    J(X.\phi) = X.J(\phi).
\end{align}
We have thus defined $J$ as a linear map on $\pi = \bigoplus_{\mu} \pi_\mu$. We must now argue that $J$ commutes with the action of $\gC^\prime$. This is recorded in the following proposition:
\begin{prop} \label{prop: g module map}
    Let $J$ be as just defined, and let $\phi \in \pi$ be arbitrary. Then, for any $Z \in \gC^\prime$ we have $J(Z.\phi) = Z.J(\phi)$.
\end{prop}
\begin{proof}
    We can assume without loss of generality that $\phi \in \pi_\mu$ is a weight vector, since the weight vectors generate $\pi$ as a vector space.
    We need to prove the following three statements:
    \begin{align}
        \label{eq: J commutes with X}
        J(X.\phi)&=X.J(\phi)
        \\
        \label{eq: J commutes with H}
        J(H.\phi)&=H.J(\phi)
        \\
        \label{eq: J commuteswith Y}
        J(Y.\phi)&=Y.J(\phi).
    \end{align}
    Note that if $\mu \not\in \beta-\ZZ\alpha$ then all three identities are automatically true since $J(\pi_{\mu}), J(\pi_{\mu\pm\alpha})$ would all vanish and $\gC^\prime \pi_\mu \subset \pi_{\mu}\oplus \pi_{\mu-\alpha}\oplus\pi_{\mu+\alpha}$. We can therefore assume $\mu = \beta-r\alpha$ for some integer $r$.

    Now let us start with \eqref{eq: J commutes with X}. We must verify that $J(X\phi)=XJ(\phi)$. Note that if $r<-1$ then $J(X \pi_\mu)=0 = X.J(\pi_\mu)$. The case for $r\geq 0$ follows from the inductive definition of $J$ given in \eqref{eq: inductive definition of J}. Hence, the remaining case for \eqref{eq: J commutes with X} is when $r=-1$, ie. $\mu=\beta+\alpha$. In this case, recall from formula
    \eqref{eq: highest weight vector} that
    \begin{align}
        \phi_\beta &= X_{\beta} e_{0,0}(0) = \lambda_0 e_{1,0}(0) \text{ and }
        \\
        Y\phi_\beta &=Y X_{\beta} e_{0,0}(0) = \lambda_1 e_{1,1}(-1),
    \end{align}
    for nonzero scalars $\lambda_0, \lambda_1\in \CC^\times$. 
    Note that $e_{1,0}(0)$ and $e_{1,1}(-1)$ are both highest weight vectors annihilated by $X_{\alpha+\beta}$ as shown in Theorem \ref{thm: action of root vectors}. Therefore, the same must hold for $\phi_\beta$ and $Y\phi_{\beta}$. By Lemma 
    \ref{lem: GGP root string}
    \eqref{sublem: GGP root string PBW basis}, 
    we get that
    \begin{align}
        \{X^{k} Y^{k+1} \phi_{\beta}\}_{k\geq 0} \subset \pi_{\alpha+\beta}
    \end{align}
    is a basis of the $\alpha+\beta$ weight space $\pi_{\alpha+\beta}$. We now take this basis, and apply $X$ to all of its elements to obtain
    \begin{align}
        \{X^{k+1} Y^{k+1} \phi_\beta\}_{k\geq 0}\subset \pi_{\beta}.
    \end{align}
    Now note that $\{X^{k+1} Y^{k+1} \phi_\beta\}_{k\geq 0}$ generates the same vector subspace as $\{(X_{-\alpha} X_{\alpha})^{k+1} \phi_\beta\}_{k\geq 0}$ due to Lemma
    \ref{lem: GGP root string}
    \eqref{sublem: GGP root string PBW basis}
    \eqref{eq: span from 1}. This means that $J$ vanishes on $\{X^{k+1} Y^{k+1} \phi_\beta\}_{k\geq 0}$, since it was defined to vanish on $\{(X_{-\alpha} X_{\alpha})^{k+1} \phi_\beta\}_{k\geq 0}$. Putting everything together, we have verified that
    \begin{align}
        J(X (X^{k} Y^{k+1} .\phi_\beta)) = 0 = X. J(X^k Y^{k+1}.\phi_\beta)
    \end{align}
    on every basis element $X^k Y^{k+1}. \phi_\beta \in \pi_{\beta+\alpha}$, which means we have verified \eqref{eq: J commutes with X} in the final case of $r=-1$. This completes the proof of \eqref{eq: J commutes with X}.
    
    We then need to prove \eqref{eq: J commutes with H}. This follows from Proposition \ref{prop: restriction of weights}, where we have
    \begin{align*}
        (\beta-r\alpha)|_{\tC^\prime} = -(r+\frac{1}{2})\alpha^\prime \text{ for all } r\geq 0.
    \end{align*}
    This implies that $J(\pi_\mu) \in \pi^\prime_{\mu|_{\tC^\prime}}$ whenever $J(\pi_\mu)$ is nonzero. Since $J$ is linear and $\tC^\prime$ acts on each weight vector by scalars, \eqref{eq: J commutes with H} follows.

    It remains to show \eqref{eq: J commuteswith Y}. We do this by taking $\phi \in \pi_{\beta-r\alpha}$ and proving \eqref{eq: J commuteswith Y} by induction on $r\geq 0$. For the base case, note that if $\phi \in \pi_{\beta-r\alpha}$ for $r=0$ then $Y\phi \in \pi_{\beta+\alpha}$, on which $J$ was defined to be zero. Now suppose by way of induction that \eqref{eq: J commuteswith Y} holds for $r$, and let $\phi \in \pi_{\beta-(r+1)\alpha}$. 
    By Lemma
    \ref{lem: GGP root string}
    \eqref{sublem: GGP root string isomorphism}, 
    there exists some $\xi \in \pi_{\beta-r\alpha}$ such that $X \xi = \phi$. By \eqref{eq: J commutes with X} and \eqref{eq: J commutes with H}, and the definition $H=[X,Y]$, we have
    \[
        J(Y.\phi)
        =J(YX.\xi)=J(-H.\xi)+J(XY. \xi)=-H.J(\xi)+ X.J(Y.\xi)
    \]
    Since $\xi \in \pi_{\beta-r\alpha}$ we can apply the induction hypothesis and complete this expression into
    \begin{align*}
        J(Y.\phi)
        &=-H.J(\xi)+ X.J(Y.\xi)
        \\
        &=-H.J(\xi)+ XY.J(\xi)
        \\
        &=(-H+XY).J(\xi) = YX.J(\xi) = Y.J(X\xi) = Y.J(\phi).
    \end{align*}
    This completes the proof of the proposition.
\end{proof}
By the above proposition, we have that $J$ is a morphism of $\gC^\prime$ modules. Combining Lemma \ref{lem: nonzero cohomology restriction implies LDS} and Proposition \ref{prop: g module map}, we have thus shown:
\begin{thm}\label{Main Thm: Low rank GGP}
    Let $\pi$ be the Harish-Chandra module of the unique TDLDS of $SU(2,1)$. Let $\pi^\prime$ be any admissible irreducible Harish-Chandra module for $SU(1,1)$.
    Let $\gC^\prime$ be the complexification of $\mathrm{Lie}(SU(1,1))$. Let $\nn$ be the negative root space for $SU(2,1)$. Let $\nn^\prime=\nn\cap \gC^\prime$, also considered a negative root space for $SU(1,1)$.
    
    We can choose $\nn$ and the embedding $SU(1,1)\subset SU(2,1)$ such that the following are equivalent: 
    \begin{itemize}
        \item there exists a $\gC^\prime$ module homomorphism $J:\pi\to \pi^\prime$ such that the induced map on cohomology
        \begin{align*}
            H^*(\nn,\pi)\to H^*(\nn^\prime,\pi^\prime)
        \end{align*}
        is non-zero.
        \item 
        $\pi^\prime$ is the the anti-holomorphic limit of discrete series for $\nn^\prime$, ie. the limit of discrete series with negative weights.
    \end{itemize}
\end{thm}



\bibliography{refs}{}

@inbook{bgg1975,
title = "Differential operators on the base affine space and study of g-modules",
author = "Joseph Bernstein and Gelfand, {I. M.} and Gelfand, {S. I.}",
note = "Includes bibliographies Summer School on Group Representations : Bolyai Janos Mathematical Society (1971 Budapest) ",
year = "1975",
language = "אנגלית",
isbn = "0 82754 296 7",
pages = "21--64",
editor = "Gelʹfand, {I. M. (Izrailʹ Moiseevich)}",
booktitle = "Lie groups and their representations",
publisher = "London : Hilger",
}

@article{carayol1998, 
    title={Limites dégénérées de séries discrètes, formes automorphes et variétés de Griffiths–Schmid: le cas du groupe U (2, 1)}, 
    volume={111}, 
    DOI={10.1023/A:1000282229017}, 
    number={1}, 
    journal={Compositio Mathematica}, 
    publisher={London Mathematical Society}, 
    author={Carayol, Henri}, 
    year={1998}, 
    pages={51–88}
}

@article{co1975,
    author = {Casselman, William and Osborne, M. Scott},
    title = {The $n$-cohomology of representations with an infinitesimal character},
    journal = {Compositio Mathematica},
    pages = {219--227},
    publisher = {Noordhoff International Publishing},
    volume = {31},
    number = {2},
    year = {1975},
    zbl = {0343.17006},
    mrnumber = {396704},
    language = {en},
    url = {http://www.numdam.org/item/CM_1975__31_2_219_0/}
}

@book{fhw2006,
  author    = {Fels, Gregor and Huckleberry, Alan T. and Wolf, Joseph A.},
  title     = {Cycle Spaces of Flag Domains: A Complex Geometric Viewpoint},
  series    = {Progress in Mathematics},
  volume    = {245},
  publisher = {Birkhäuser},
  year      = {2006},
  isbn      = {978-0-8176-4391-1},
  doi       = {10.1007/0-8176-4531-6},
  url       = {https://doi.org/10.1007/0-8176-4531-6}
}

@book{ggk2013,
    author = {Green, Mark. and Griffiths, Phillip. and Kerr, Matt.},
    title = {Hodge Theory, Complex Geometry and Representation Theory},
    isbn = { 9871470410124 },
    publisher = { American Mathematical Society Providence, RI },
    pages = { 197 p. ; },
    year = { 2013 },
    type = { Book },
    language = { English },
}

@misc{grt2013,
      title={Quotients of non-classical flag domains are not algebraic}, 
      author={Phillip Griffiths and Colleen Robles and Domingo Toledo},
      year={2013},
      eprint={1303.0252},
      archivePrefix={arXiv},
      primaryClass={math.AG}
}

@misc{harris2012,
      title={The Refined Gross-Prasad Conjecture for Unitary Groups}, 
      author={R. Neal Harris},
      year={2012},
      eprint={1201.0518},
      archivePrefix={arXiv},
      primaryClass={math.NT},
      url={https://arxiv.org/abs/1201.0518}, 
}

@misc{hks2025,
      title={Translation functors, branching problems, and applications to the restriction of coherent cohomology of Shimura varieties}, 
      author={Michael Harris and Toshiyuki Kobayashi and Birgit Speh},
      year={2025},
      eprint={2509.17007},
      archivePrefix={arXiv},
      primaryClass={math.RT},
      url={https://arxiv.org/abs/2509.17007}, 
}

@inbook{humphreys1990,
    place={Cambridge}, 
    series={Cambridge Studies in Advanced Mathematics}, 
    title={Reflection Groups and Coxeter Groups}, 
    publisher={Cambridge University Press}, 
    author={Humphreys, James E.}, 
    year={1990}, 
    collection={Cambridge Studies in Advanced Mathematics}}

@article{jw1977,
 ISSN = {00029947},
 URL = {http://www.jstor.org/stable/1998503},
 author = {Kenneth D. Johnson and Nolan R. Wallach},
 journal = {Transactions of the American Mathematical Society},
 pages = {137--173},
 publisher = {American Mathematical Society},
 title = {Composition Series and Intertwining Operators for the Spherical Principal Series. I},
 urldate = {2026-04-15},
 volume = {229},
 year = {1977}
}

@article{soergel1997,
  title={On the n-cohomology of limits of discrete series representations},
  author={Wolfgang Soergel},
  journal={Representation Theory of The American Mathematical Society},
  year={1997},
  volume={1},
  pages={69-82}
}

@article {Carayol_Knapp,
    AUTHOR = {Carayol, Henri and Knapp, A. W.},
     TITLE = {Limits of discrete series with infinitesimal character zero},
   JOURNAL = {Trans. Amer. Math. Soc.},
  FJOURNAL = {Transactions of the American Mathematical Society},
    VOLUME = {359},
      YEAR = {2007},
    NUMBER = {11},
     PAGES = {5611--5651},
      ISSN = {0002-9947},
   MRCLASS = {22E45 (14L35 20G20)},
  MRNUMBER = {2327045},
MRREVIEWER = {David A. Renard},
       DOI = {10.1090/S0002-9947-07-04306-1},
       URL = {https://doi.org/10.1090/S0002-9947-07-04306-1},
}

@misc{mcgovern2012,
      title={Closures of K-orbits in the flag variety for U(p,q)}, 
      author={William M. McGovern},
      year={2012},
      eprint={0905.0127},
      archivePrefix={arXiv},
      primaryClass={math.RT}
}

@article{williams1988,
    title = {The n-cohomology of limits of discrete series},
    journal = {Journal of Functional Analysis},
    volume = {80},
    number = {2},
    pages = {451-461},
    year = {1988},
    issn = {0022-1236},
    doi = {https://doi.org/10.1016/0022-1236(88)90011-0},
    url = {https://www.sciencedirect.com/science/article/pii/0022123688900110},
    author = {Floyd L Williams},
}

@incollection {GGP,
    AUTHOR = {Gan, Wee Teck and Gross, Benedict H. and Prasad, Dipendra},
     TITLE = {Symplectic local root numbers, central critical {$L$} values,
              and restriction problems in the representation theory of
              classical groups},
      NOTE = {Sur les conjectures de Gross et Prasad. I},
   JOURNAL = {Ast\'{e}risque},
  FJOURNAL = {Ast\'{e}risque},
    NUMBER = {346},
      YEAR = {2012},
     PAGES = {1--109},
      ISSN = {0303-1179},
      ISBN = {978-2-85629-348-5},
   MRCLASS = {22E50 (11F70 11R39 22E55)},
  MRNUMBER = {3202556},
}
\bibliographystyle{alpha}

\end{document}